\documentclass{article}
\usepackage{amsmath,amstext,amsthm,amsfonts}
\usepackage[dvips]{graphicx}
\usepackage{amssymb}
\newtheorem{teorema}{Theorem}
\newtheorem{proposition}[teorema]{Proposition}
\newtheorem{remark}[teorema]{Remark}

\newtheorem{lema}[teorema]{Lemma}
\newtheorem{definicao}[teorema]{Definition}
\newtheorem{corolario}[teorema]{Corollary}

\newcommand{\EE}{\mathbb{E}}
\newcommand{\VV}{\mathbb{V}}

\newcommand{\ZZ}{\mathbb{Z}}
\newcommand{\NN}{\mathbb{N}}

\newcommand{\wh}{\widehat}
\newcommand{\al}{\alpha}
\newcommand{\wt}{\widetilde}
\newcommand{\om}{\omega}
\newcommand{\Up}{\Upsilon}
\theoremstyle{plain}
\begin{document}

\title{{{On the compatibility of  binary sequences}}}
\author{H. Kesten\footnote{Department of Mathematics, 310 Malott Hall, Cornell University, Ithaca, NY 14853-4201 USA.  email:
\textsf{hk21@cornell.edu}} ,   B.N.B. de Lima\footnote{UFMG, Av. Ant\^onio Carlos 6627,
CEP 30123-970, Belo Horizonte, MG, Brasil. \newline email:
\textsf{bnblima@mat.ufmg.br}} , \ V. Sidoravicius\footnote{IMPA,
Estrada Dona Castorina 110, CEP 22460-320, Rio de Janeiro, RJ, Brasil.
email: \textsf{vladas@impa.br}} ,  M. E. Vares\footnote{IM-UFRJ, Av. Athos da Silveira Ramos 149, CEP 21941-909 - Rio de Janeiro, RJ, Brasil.  email:
\textsf{eulalia@cbpf.br}} , } \maketitle
\begin{abstract} An ordered pair of semi-infinite binary sequences $(\eta,\xi)$ is said
to be compatible if there is a way of removing a certain number
(possibly infinite) of ones from $\eta$ and zeroes from $\xi$,
which
would map both sequences to the same semi-infinite sequence. This notion was introduced by  Peter
Winkler, who also posed the following question: $\eta$ and $\xi$ being
independent i.i.d. Bernoulli sequences with parameters $p^\prime$ and
$p$ respectively, does it exist  $(p', p)$ so that the set
of compatible pairs has positive measure?  It is known that this
does not happen for $p$ and $p^\prime$ very close to $1/2$. In the positive direction, we construct, for any $\epsilon > 0$,
a deterministic binary sequence $\eta_\epsilon$ whose set of zeroes has Hausdorff
dimension larger than $1-\epsilon$, and
such that
$\mathbb{P}_p
\{\xi\colon (\eta_\epsilon,\xi) \text { is compatible}\} > 0$ for $p$ small
enough, where $\mathbb{P}_p$ stands for the product Bernoulli
measure with parameter $p$.

\end{abstract}

\textbf{Keywords:} compatibility of sequences, dependent percolation.

\section{Introduction\label{I}}

Consider the set of all semi-infinite binary sequences
$\Xi=\{0,1\}^{\NN}$, with $\NN=\{1,2,\dots\}$. For each $i\in\NN$
define the ``annihilation" operators $\triangle_i^{\mathbf{0}}$ and
$\triangle_i^{\mathbf{1}}:\Xi\rightarrow\Xi$, which act according to
the following rules:  if $\xi=(\xi_1,\xi_2,\dots) \in \Xi$ and
$\xi_i=0$, then
\begin{equation}
\label{annihil0} \triangle_i^{\mathbf{0}}(\xi)_j=
\begin{cases}
\xi_{j+1}, &\mbox{ if }j\geq i,
\\
\xi_j, &\mbox{ if }j< i,
\end{cases} \qquad \quad \text{and} \qquad \quad   \triangle_i^{\mathbf{1}} (\xi) = \xi;
\end{equation}
and, respectively, if $\xi_i=1$, then
\begin{equation}
\label{annihil1} \triangle_i^{\mathbf{0}} (\xi) = \xi, \qquad \quad
\text{and} \qquad \quad
 \triangle_i^{\mathbf{1}}(\xi)_j=
\begin{cases}
\xi_{j+1}\ ,&\mbox{ if }j\geq i,
\\
\xi_j\ , &\mbox{ if }j< i.
\end{cases}
\end{equation}
In other words, for every $ i\in \NN$, the sequence
$\triangle_i^{\mathbf{0}}(\xi)$, respectively
$\triangle_i^{\mathbf{1}} (\xi)$, is obtained from the sequence
$\xi$ by annihilating (deleting) the $i$-th digit if it is 0,
respectively 1, and shifting all elements of $\xi$ which are to the
right of the $i$-th position by one unit to the left. We consider $\Xi$
as metric space with the usual product topology.

\begin{definicao}
\label{compat-binary}  Let $\eta=(\eta_1,\eta_2,\dots)$ and
$\xi=(\xi_1,\xi_2,\dots)$ be elements of $\Xi$. The pair
$(\eta,\xi)$ is said to be compatible if there exist sequences
$(\eta^{(k)})_{k\ge 1}$ and $(\xi^{(k)})_{k\ge 1}$ in $\Xi$ such
that: each $\eta^{(k)}$ is obtained from $\eta$ by finitely many
applications of operators $\triangle_{\cdot}^{\mathbf{1}}$; each
$\xi^{(k)}$ is obtained from $\xi$ by finitely many applications of
operators $\triangle_{\cdot}^{\mathbf{0}}$; as $k \to \infty$, both
sequences converge in $\Xi$ to the same limit.
\end{definicao}

\noindent Informally speaking, the pair $(\eta,\xi)$ is compatible
if by deleting some ones in the first sequence and deleting some
zeroes in the second sequence one can make them equal.

\medskip

\noindent For $p\in [0,1]$, let $\mathbb P_p$ be the probability product measure
on $\Xi = \{0,1 \}^{\mathbb N}$ such that $\mathbb P_p (\xi_i=1) = p$  for all $i \ge 1$.
Motivated by scheduling problems, P. Winkler (see \cite{Wi}) posed the following question:

\medskip

\noindent {\it Does it exist a pair $(p, p') \in (0,1)^2$
such that}
\begin{equation*}
\mathbb{P}_p \otimes \mathbb P_{p'} \{(\eta,
\xi)\in\Xi\times\Xi\colon (\eta,\xi) \text{ are compatible} \}
>0\,?
\end{equation*}

\medskip
\noindent This question has been addressed in \cite{Gac}.
A simple form of Peierls argument (see \cite{Wi})
shows that if $p$ and $p'$ are close to 1/2, then
\begin{equation*}
\mathbb{P}_p \otimes \mathbb P_{p'}\{(\eta,\xi)\in\Xi\times\Xi\colon
(\eta,\xi)\mbox{ is compatible}\}=0.
\end{equation*}

In this context, it is then natural to introduce the following

\begin{definicao}
For $p \in [0,1]$ we say that $\eta \in \Xi$ is $p$-compatible in $\Xi$, or $p$-compatible for short, if
\begin{equation}
\label{compatibilidade} \mathbb{P}_p\{\xi\in \Xi\colon (\eta, \xi)
\text{ is compatible}\}>0.
\end{equation}
\end{definicao}

\medskip

\noindent For $\underline{\mathbf{1}} := (1,1, 1,
\dots)$, the pair $(\underline{\mathbf{1}},\xi)$ is compatible
as long as the sequence $\xi$ has infinitely many ones.
Thus, for all $p>0$
\begin{equation}
\label{25aug1} \mathbb{P}_p\{\xi\in \Xi\colon
(\underline{\mathbf{1}},\xi) \text{ is compatible}\}=1,
\end{equation}
and if $\underline{\mathbf{1}}$ is replaced by $\eta \in \Xi$ which has only finitely
many zeroes, then for every $0< p < 1$
\begin{equation}
\label {25aug2} \mathbb{P}_p\{\xi\in \Xi\colon (\eta,\xi) \text{ is
compatible}\}>0.
\end{equation}
If $\eta$ has infinitely many zeroes but the distance between consecutive
zeroes increases fast enough, for instance exponentially fast with a rate that suitably depends on $p$, it is
straightforward to see that  \eqref{25aug2}
still holds. However it is a priori unclear whether there are deterministic binary
sequences $\eta$ with a richer (and more complicated) set of zeroes $\mathcal{Z}_\eta := \{ i\ge 1: \; \eta_i =0 \}$
which still can be $p$-compatible for some positive $p$.
Our main result can be stated as follows:

\begin{teorema}
\label{principal}
For every $\epsilon >0$ there exist $0< p_{\epsilon} <1$  and
a binary sequence $\eta\equiv \eta_{\epsilon} \in\Xi$, such that $\mathcal{Z}_{\eta_{\epsilon}}$ is a discrete fractal in the sense of \cite{BT}, with the
Hausdorff dimension $d_H(\mathcal{Z}_\eta) \ge 1-\epsilon$, and such that
\begin{equation*}
\mathbb{P}_p\{\xi \in \Xi \colon (\eta,\xi) \text{ is
compatible}\}>0
\end{equation*}
for any $p < p_{\epsilon}$.
\end{teorema}

\medskip

\noindent For the proof of Theorem \ref{principal} it is convenient to exploit
a representation of compatibility of binary sequences in the language of dependent oriented two dimensional percolation. In Section \ref{II} we
describe the percolation model defining its configuration in terms of $(\eta, \xi)$,  and give conditions on  $\eta$ and $\xi$ which guarantee the existence of an infinite open path from the origin.

The key ingredient of the proof consists in showing that Bernoulli sequences $\xi$ with small density of ones can be suitably mapped
by appropriate grouping of ones and removing unwanted zeroes into  sequences that satisfy the conditions mentioned in the previous paragraph.
This is done in Sections \ref{III} and \ref{IV} together with the explicit construction of a deterministic binary sequence $\eta$. The proof of Theorem \ref{principal} is completed in Section \ref{IV}. Details of the grouping lemma are given in the Appendix.

\section{Percolation process \label{II}}

In this section we describe the related dependent percolation model.
For this we first introduce an alternative representation of binary
sequences which is useful for our purposes.

Consider the set $\Xi_\infty \subset \Xi$ of all binary sequences
$\xi \in \Xi$ that contain infinitely many ones and infinitely many
zeroes:
$$
\Xi_\infty=\big\{\xi \in \{0,1 \}^{\NN} \colon  \; |\{i: \xi_i \neq
\xi_{i+1}\}| =  \infty\big\}.
$$
Each $\xi \in \Xi_\infty$ can be represented as an element $f(\xi)$
of $\mathbb{Z}_+^{\NN}$, where each run of ones in $\xi$ is replaced
by a single coordinate whose value is the cardinality (length) of
the run, with the corresponding shift to the left of the part of the
sequence that follows the run. More precisely, define
$$
\Psi=\big\{\psi \in \mathbb{Z}_+^{\NN}\colon \psi_i \ge 1 \text{
implies } \psi_{i+1}=0\big\}.
$$

\noindent An element $\psi \in \Psi$   will be called a
weighted word and the value $\psi_j\in\NN$ will be called the weight of the
$j$-th letter of $\psi$. Given $\xi\in \Xi_\infty$,  the sequence $f(\xi)$ can be defined recursively:

\begin{equation} \label{f1}
\big(f(\xi)\big)_1 =
\begin{cases}
k,   & \mbox{ if } \xi_{s} = 1, \mbox{ for } s= 1, \dots , k, \mbox{
and } \xi_{k+1} = 0,
\\
0, & \mbox{ if } \xi_{1} = 0.
\end{cases}
\end{equation}
If $\big(f(\xi)\big)_i, \; 1 \leq i \leq j-1$ are defined, for $j
\ge 2$, next we define
\begin{equation} \label{h}
h_{j-1} (\xi) = \sum_{i=1}^{j-1} \left(\big(f(\xi)\big)_i
+{\mathbf{I}}_{[(f(\xi))_i =0]}\right),
\end{equation}
and we set
\begin{equation} \label{f2}
\big(f(\xi)\big)_j =
\begin{cases}
k,  &  \mbox{ if }\;  \xi_{s} = 1, \mbox{ for } s= h_{j-1} (\xi) + r, \; r= 1, \dots , k,  \\
  &  \qquad \qquad \; \; \, \mbox{ and } \;  \xi_{h_{j-1} (\xi) + k+1} =0 ,
\\
0,   & \mbox{ if } \; \xi_{h_{j-1}(\xi) + 1} =0.
\end{cases}
\end{equation}
The just defined map $f\colon \Xi_\infty \to \Psi$ is one-to-one,
and so defines a bijection between $\Xi_\infty$ and the subset of
$\Psi$ of weighted words which have infinitely many non-zero
entries.


\medskip

\noindent For each $i\in\NN$ define the ``annihilation" operators
$\triangle_i^{\mathbf{0}}$ and
$\triangle_i^{\mathbf{1}}:\mathbb{Z}_+^{\NN}\rightarrow\mathbb{Z}_+^{\NN}$,
which act according to the following rules:  if
$\psi=(\psi_1,\psi_2,\dots) \in \mathbb{Z}_+^{\NN}$ and $\psi_i\ge
1$, then $ \triangle_i^{\mathbf{0}} (\psi) = \psi$; otherwise and when
at least one among $\psi_{i-1}$ or $\psi_{i+1}$ is zero, or $i=1$, then
\begin{equation}
\triangle_i^{\mathbf{0}}(\psi)_j=
\begin{cases}
\psi_{j+1}, &\mbox{ if }j\geq i,
\\
\psi_j, &\mbox{ if }j< i.
\end{cases}
\end{equation}
When $i\ge 2$, $\psi_i=0$ but $\psi_{i-1} \wedge \psi_{i+1}\ge 1$, then
\begin{equation}
\triangle_i^{\mathbf{0}}(\psi)_j=
\begin{cases}
\psi_{j+2}, &\mbox{ if }j\geq i,
\\
\psi_{i-1}+\psi_{i+1},  &\mbox{ if }j= i-1,\\
\psi_j, &\mbox{ if } j< i-1.
\end{cases}
\end{equation}
Similarly, if $\psi_i= 0$, then $\triangle_i^{\mathbf{1}} (\psi) =
\psi$. Otherwise, when $\psi_i=1$ we set
\begin{equation}
\label{25aug3}
\triangle_i^{\mathbf{1}}(\psi)_j=
\begin{cases}
\psi_{j+1}, &\mbox{ if } j\geq i,
\\
\psi_j,  &\mbox{ if }j < i,
\end{cases}
\end{equation}
while when  $\psi_i>1$ we set
\begin{equation}
\label{25aug4}
 \triangle_i^{\mathbf{1}}(\psi)_j=
\begin{cases}
\psi_{i}-1, &\mbox{ if } j=i,
\\
\psi_j, &\mbox{ if }j \neq i.
\end{cases}
\end{equation}

\begin{definicao} \label{compat-weight} Let $\zeta=(\zeta_1,\zeta_2,\dots)$ and $\psi=(\psi_1,\psi_2,\dots)$ be
elements of $\Psi$. The pair $(\zeta,\psi)$ is said to be compatible
if there exist there exist sequences $(\zeta^{(k)})_{k\ge 1}$ and
$(\psi^{(k)})_{k\ge 1}$ in $\Psi$ such that: each $\zeta^{(k)}$ is
obtained from $\zeta$ by finitely many applications of operators
$\triangle_{\cdot}^{\mathbf{1}}$; each $\psi^{(k)}$ is obtained from
$\psi$ by finitely many applications of operators
$\triangle_{\cdot}^{\mathbf{0}}$; as $k \to \infty$, both sequences
have a common limit in $\Psi$ (product topology).
\end{definicao}


\noindent The following proposition follows at once from the definitions.

\begin{proposition}\label{pesos2} Let $(\eta,\xi)$ be a pair of configurations in $\Xi_\infty$, and let $\zeta=f(\eta)$,
$\psi=f(\xi)$ with $f$ the map defined in \eqref{f1}--\eqref{f2}. If the pair
$(\zeta,\psi)$ is compatible, then so is $(\eta,\xi)$.
\end{proposition}

\medskip


\noindent {\bf Percolation process.} The percolation process will be
defined on the oriented graph $\mathcal{G}=(\VV,\EE)$, where
$\VV=\ZZ_+^2$ and $\EE=\{\langle v,w \rangle; w_1=v_1\text{ and }
w_2=v_2+1;  w_1=v_1+1\text{ and } w_2=v_2+1\}$, with $v=(v_1,v_2),
w=(w_1,w_2)$, i.e. the edges consist of the nearest neighbor vertical edges and
the northeast oriented diagonals on the first quadrant of $\ZZ^2$;
all the edges $\langle v,w \rangle$ are oriented from $v$ to $w$.

\medskip

\noindent Given two sequences $ \zeta,\, \psi \in \Psi$, we define
the north-east oriented site percolation configuration
$\omega_{\zeta,\psi}$ on $\mathcal{G}$ by setting:
$\omega_{\zeta,\psi}(0,0)=1$, for $v=(v_1,v_2)\in
\VV\setminus\{(0,0)\}$ set $\omega_{\zeta,\psi}(v)=0$ if $v_1 \wedge
v_2=0$, while for $v_1,v_2 \ge 1$:
\begin{equation}
\label{25aug5}
\omega_{\zeta,\psi} (v) = \begin{cases} 1 \quad &\text { if } \zeta_{v_1} \geq \psi_{v_2}, \\
0 & \text { otherwise.} \end{cases}
\end{equation}
We say that $v$ is {\it open} if $\omega_{\zeta,\psi} (v)=1$. An
oriented path $\pi:=\{u=v^{(0)},e^{(1)},v^{(1)},\break\dots, e^{(n)},$ $v^{(n)}=v \}$ from
$u$ to $v$ is said to be open if each vertex $v^{(i)}$ is open (here $e^{(i)}=\langle
v^{(i-1)},v^{(i)}\rangle \in \EE$ for each $i$). We say that $v\in \VV$ belongs to the
open oriented cluster of a vertex $u$
if there is an oriented path. The open oriented
cluster of the origin is denoted by ${\mathcal{C}}_{\langle \zeta,
\psi \rangle} \equiv {\mathcal{C}}_{\langle \zeta, \psi \rangle}
(\mathbf{0}) $. More generally, given any finite subset $I \subset
\ZZ_+^2$, by ${\mathcal{C}}_{\langle \zeta, \psi \rangle} (I)$ we
denote the open oriented cluster of $I$, i.e. the union of the open
oriented clusters of the vertices in $I$.

\begin{figure}[tbpp]
\centering \hskip .5cm
\includegraphics[scale=0.25]{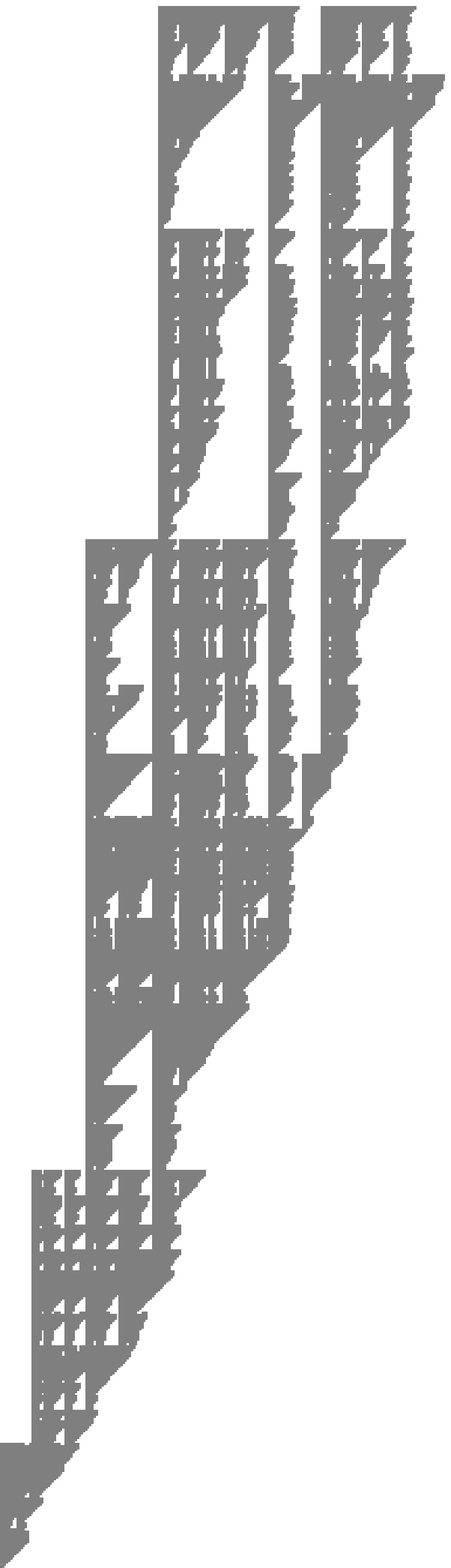}
\includegraphics[scale=0.25]{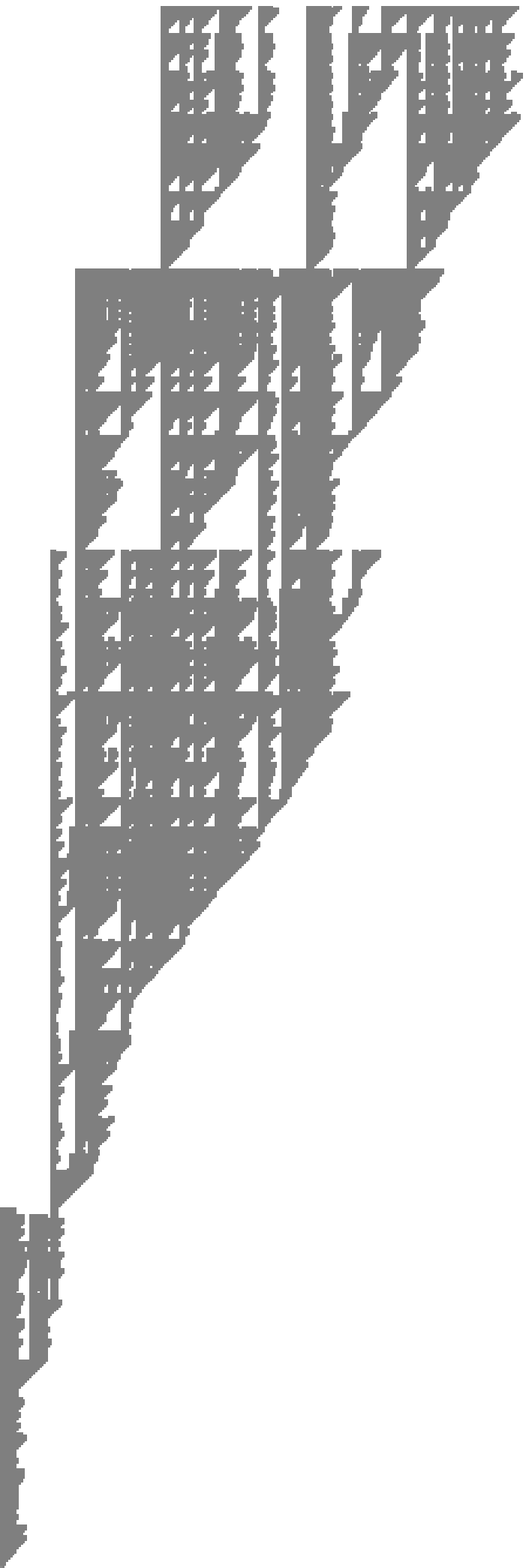}
\includegraphics[scale=0.25]{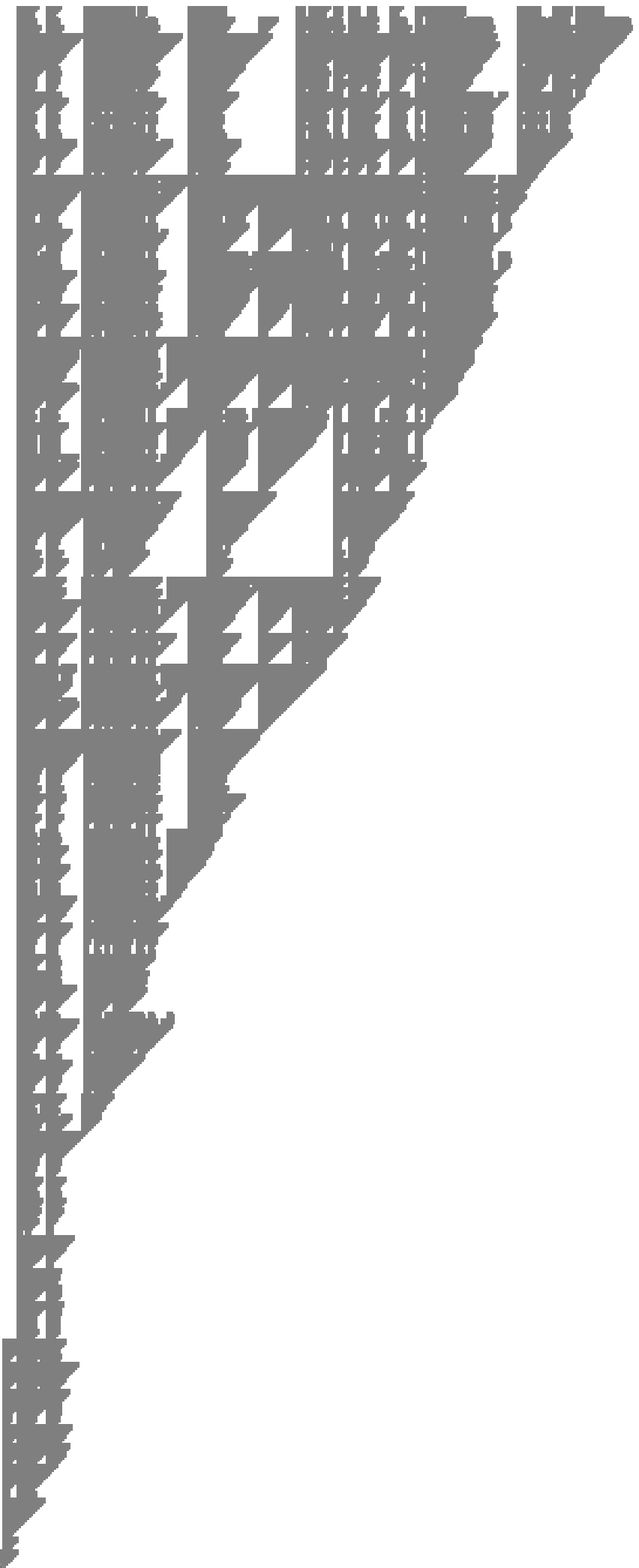}
\includegraphics[scale=0.25]{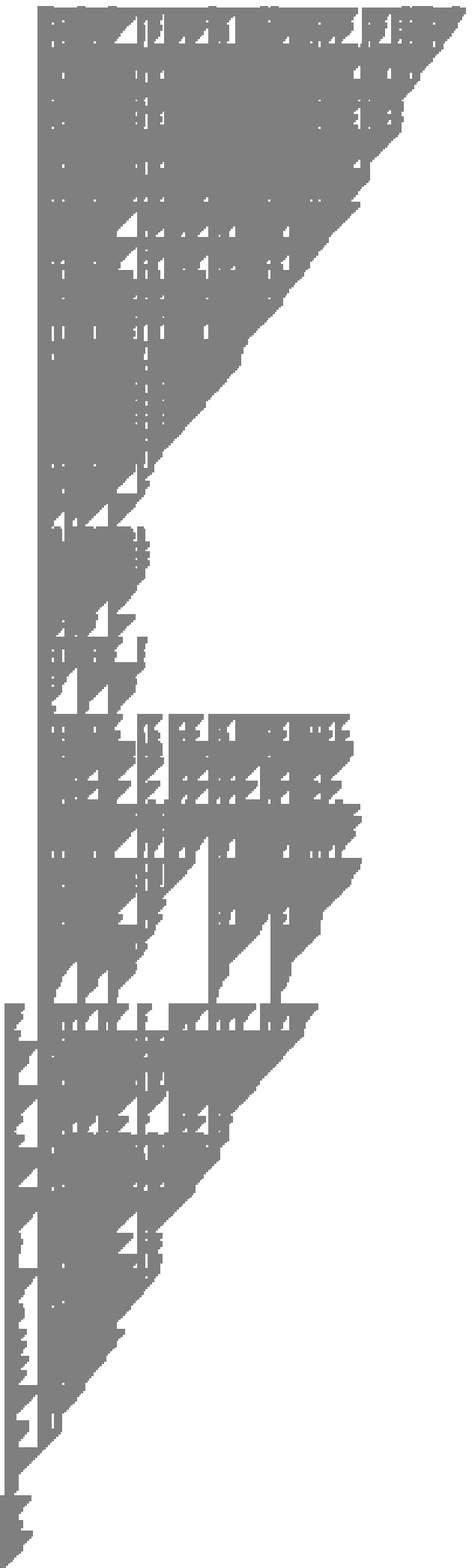}
\caption{Simulated samples of the open oriented cluster for pair of 
geometric variables $(\zeta,\psi)$. 
Unusually large values along the $x$-axis create long vertical open 
segments, while large values along the $y$-axis produce drastic horizontal cuts.}
\label{fig1}
\end{figure}

\begin{definicao} \label{percolacao} A vertex $v = (v_1, v_2) \in \ZZ_+^2$ is called ``heavy'' if
$v_2 = 0$ or, when $v_2 \ge 1$, if $\psi_{v_2} > 0$.
\end{definicao}

\begin{definicao} \label{permit}
Given two weighted words $\zeta, \; \psi \in \Psi$ and the
associated percolation configuration $\omega_{\zeta,\psi}$, we say
that an infinite oriented path $\pi$ starting from the origin is
 {\it permitted} if for any pair of heavy vertices $u
= (u_1, u_2)$ and $v= (v_1, v_2)$ in $\pi$ we have $u_1 \neq v_1$.
\end{definicao}

\begin{figure}[tbp]
\centering
\includegraphics[scale=0.60]{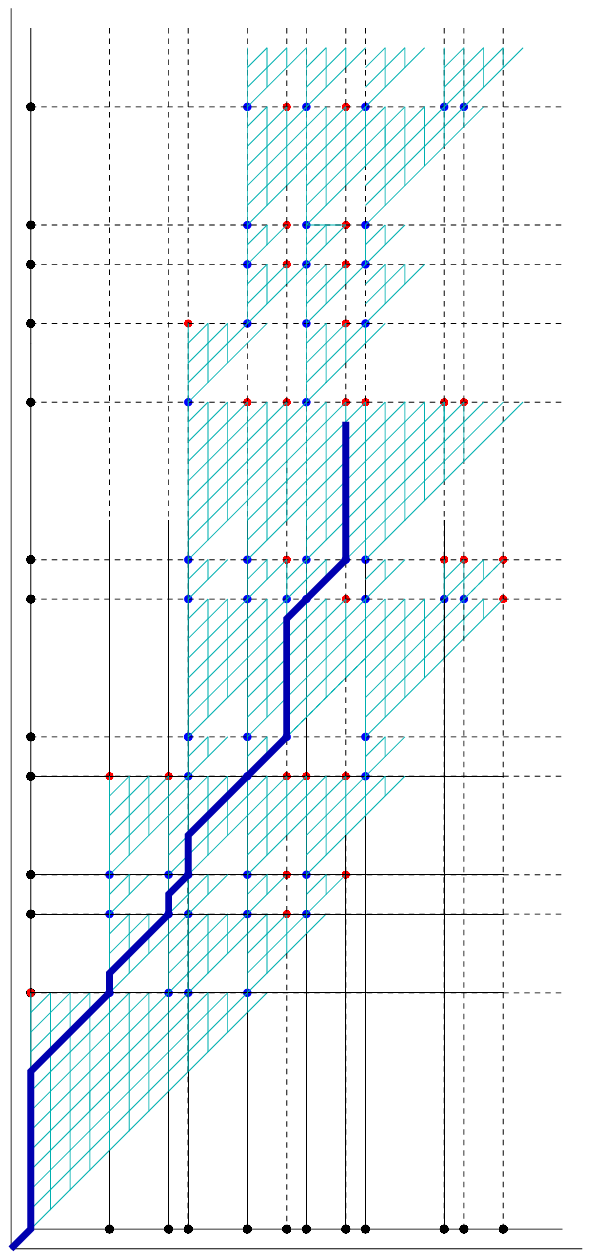}
\caption{Illustration of the open oriented cluster of the origin. Non-zero entries marked along the axes. Marked in bold a longest open permitted path.}
\label{fig1}
\end{figure}

\begin{lema}
\label{compatibility-lemma} Let $\zeta, \; \psi \in \Psi$. If there
exists an infinite open permitted path $\pi$ starting from the
origin for the percolation configuration $\omega_{\zeta, \psi}$,
then the pair $(\zeta,\psi)$ is compatible.
\end{lema}

\proof Assume that $\zeta$ and $\psi$ satisfy the
conditions in the statement, and let $\pi$ be an infinite open
permitted path starting from the origin. Let us first assume that $\psi$
contains infinitely many non-zero
entries. In this case, we consider the increasing sequence of
indices $\{\ell_j\}_{j=1}^{+\infty}$ that correspond to all non-zero
entries of $\psi$, i.e., $\psi_{i}\ge 1$ if and only if $i=\ell_j$
for some $j$. In particular $\ell_{j+1}> \ell_j +1$ for all $j\ge
1$. Let $\{ v_{\ell_j} = (x_{\ell_j}, \ell_j)\}_{j=1}^{+\infty}$ be
the corresponding sequence of heavy open vertices in
$\pi\setminus\{(0,0)\}$, and set $\ell_0=0, x_0=0, v^{(0)}=(0,0)$. Since $\pi$
is permitted, we have that $x_{\ell_j} < x_{\ell_{j'}}$ if $j < j'$.
Since $v_{\ell_j}$ is open we also have that $\zeta_{x_{\ell_j}}
\geq \psi_{\ell_j}$ for all $j \ge 1$. If $\pi'$ is an arbitrary
oriented path on $\mathcal{G}$, and $v = (v_1, v_2)$ and $u = (u_1,
u_2)$  are on $\pi'$, then $|u_2 - v_2| \ge |u_1 -v_1|$. In
particular for the open permitted path $\pi$ we have ${\ell_{j}} -
{\ell_{j-1}  \geq {x_{\ell_{j}}} - x_{\ell_j-1}}\ge 1$ for all $j
\ge 1$.

Denote by $j_{k,1} < \dots  < \; j_{k,m_k}$ the indices of non-zero
entries of $\zeta$ which lie strictly between $x_{\ell_{k-1}}$ and
$x_{\ell_{k}}$. This set could be empty, in which case set $m_k=0$.
Notice that $x_{\ell_{k}} - \sum_{s=1}^k m_s$ is strictly increasing
in $k$.
\medskip

Define:

\begin{equation}
 \zeta^{(1)} := \begin{cases} \big( [\triangle^{\mathbf{1}}_{j_{1,1}}]^{\zeta_{j_{1,1}}}
  \circ \dots  \circ[\triangle^{\mathbf{1}}_{j_{1,m_1}}]^{\zeta_{j_{1,m_1}}} \big)
  \circ[\triangle^{\mathbf{1}}_{x_{\ell_{1}}}]^{\zeta_{x_{\ell_1}} -  \psi_{\ell_1}} (\zeta),  &\mbox{ if } m_1 \geq 1, \notag
 \\
 [\triangle^{\mathbf{1}}_{x_{\ell_{1}}}]^{\zeta_{x_{\ell_1}} -  \psi_{\ell_1}} (\zeta),  &\mbox{ if } m_1 =0, \notag
 \end{cases}
\end{equation}
The action of $[\triangle^{\mathbf{1}}_{x_{\ell_{1}}}]^{\zeta_{x_{\ell_1}} -
\psi_{\ell_1}}$ on $\zeta$ decreases the value of
$\zeta_{x_{\ell_{1}}}$ to $\psi_{\ell_1}$, and the action of
$[\triangle^{\mathbf{1}}_{j_{1,1}}]^{\zeta_{j_{1,1}}}  \circ \dots
\circ[\triangle^{\mathbf{1}}_{j_{1,m_1}}]^{\zeta_{j_{1,m_1}}}$ deletes all
non-zero entries of $\zeta$ which precede $\zeta_{x_{\ell_{1}}}$,
{\it i.e.}

\begin{equation}
\zeta^{(1)}_j =
\begin{cases}
0, \; &\text{if} \; j = 1, \dots , x_{\ell_{1}} - m_1 -1 \\
\psi_{\ell_1}, &\text{if} \; j = x_{\ell_{1}} - m_1 \\
\zeta_{x_{\ell_{1}} + i} &\text{if} \; j = x_{\ell_{1}} - m_1 +i,
\quad i=1,2, \dots
\end{cases} \notag
\end{equation}
Now define
\begin{equation}
\psi^{(1)} : = [\triangle^{\mathbf{0}}_1]^{\ell_1 -  x_{\ell_1}+m_1}
(\psi), \notag
\end{equation}
i.e.
\begin{equation}
\psi^{(1)}_j =
\begin{cases}
0, \; &\text{if} \; j = 1, \dots , x_{\ell_{1}} - m_1 -1 \\
\psi_{\ell_1}, &\text{if} \; j = x_{\ell_{1}} - m_1 \\
\psi_{{\ell_{1}} + i} &\text{if} \; j = x_{\ell_{1}} - m_1 +i, \quad i=1,2, \dots
\end{cases} \notag
\end{equation}
Thus, $\zeta^{(1)}_j = \psi^{(1)}_j $  for $j=
1, \dots, x_{\ell_{1}} - m_1$. One should notice that
$\psi^{(1)}$ and $\zeta^{(1)}$ are both elements
of $\Psi$; for this we recall that $\zeta_{x_{\ell_1}} \ge
\psi_{\ell_1} \ge 1$ and therefore $\zeta_{x_{\ell_1}+1}=0$, since
$\zeta \in \Psi$.

\noindent Assume to have constructed $\zeta^{(k)}$ and
$\psi^{(k)}$ in $\Psi$, and which satisfy the following
property:
\begin{align}
\zeta^{(k)}_j &=  \psi^{(k)}_j,  &{\text{if}}
\quad j \leq x_{\ell_{k}} - \sum_{s=1}^k m_s, \notag
\\
\zeta^{(k)}_j &= \zeta_{j + \sum_{s=1}^k m_s },
&{\text{if}} \quad j > x_{\ell_{k}} - \sum_{s=1}^k m_s, \label{p10}
\\
\psi^{(k)}_{j} & = \psi_{{j + \ell_{k}} + \sum_{s=1}^k m_s
 -   x_{\ell_{k}}     },   &{\text{if}}  \quad j > x_{\ell_{k}} - \sum_{s=1}^k m_s.  \notag
\end{align}
Recalling that $x_{\ell_{k}} - \sum_{s=1}^k m_s$ is strictly
increasing in $k$, we proceed as follows:

\noindent If $m_{k+1} \geq 1$, define
\begin{eqnarray*}
\zeta^{(k+1)}:=\,\big[ [\triangle^{\mathbf{1}}_{j_{k+1,1}- \sum_{s=1}^k
m_s}]^{\zeta_{j_{k+1,1}}}  \circ \dots
\circ[\triangle^{\mathbf{1}}_{j_{k+1,m_{k+1}}- \sum_{s=1}^k
m_s}]^{\zeta_{j_{k+1,m_{k+1}}}} \big] \notag
 \\
\; \circ \,[\triangle^{\mathbf{1}}_{x_{\ell_{k+1}} - \sum_{s=1}^k m_s
}]^{\zeta_{x_{\ell_{k+1}}} -  \psi_{\ell_{k+1}}}
(\zeta^{(k)})\,,
\end{eqnarray*}
\noindent while if  $m_{k+1} = 0$, set
\begin{equation*}
\zeta^{(k+1)} := \,  [\triangle^{\mathbf{1}}_{x_{\ell_{k+1}} -
\sum_{s=1}^k m_s }]^{\zeta_{x_{\ell_{k+1}}} -  \psi_{\ell_{k+1}}}
(\zeta^{(k)})\,.
\end{equation*}
\noindent On the other hand, set
\begin{equation*}
 \psi^{(k+1)}: = \,[\triangle^{\mathbf{0}}_{x_{\ell_{k}} - \sum_{s=1}^k m_s +1}]^{\ell_{k+1} -
 x_{\ell_{k+1}}+m_{k+1}} (\psi^{(k)})\,.
\end{equation*}
The action of the operator  $[\triangle^{\mathbf{1}}_{x_{\ell_{k+1}} -
\sum_{s=1}^k m_s  }]^{\zeta_{x_{\ell_{k+1}}} -  \psi_{\ell_{k+1}}}$
on $\zeta^{(k)}$ decreases the value of $
\zeta^{(k)}_{x_{\ell_{k+1}} - \sum_{s=1}^k m_s }$  from
$\zeta_{x_{\ell_{k+1}}}$ to $\psi_{\ell_{k+1} }$, and the action of
the operator $[\triangle^{\mathbf{1}}_{j_{k+1,1}- \sum_{s=1}^k
m_s}]^{\zeta_{j_{k+1,1}}} \circ \dots
\circ[\triangle^{\mathbf{1}}_{j_{k+1,m_{k+1}}- \sum_{s=1}^k
m_s}]^{\zeta_{j_{k+1,m_{k+1}}}}$ eliminates all non-zero entries of
$\zeta^{(k)}$ which lie strictly between indices
$x_{\ell_{k}} - \sum_{s=1}^k m_s$ and $x_{\ell_{k+1}} - \sum_{s=1}^k
m_s$.

As in the case $k=1$, we can check that both sequences $
\zeta^{(k+1)}$ and $\psi^{(k+1)}$ are in $\Psi$ and
satisfy properties of (\ref{p10}) with $k$ replaced by $k+1$.
Proceeding recursively, we get that the following limits exist in
$\Psi$:
$$
\zeta^{(\infty)} := \lim_{k \to \infty}
\zeta^{(k)}, \quad \text{and} \quad
\psi^{(\infty)}:= \lim_{k \to \infty} \psi^{(k)}
$$
and $\zeta^{(\infty)} = \psi^{(\infty)}$, which
implies compatibility of $\zeta$ and $\psi$.

When $\psi$
has only finitely many non-zero entries,  $\psi^{(k+1)}=\psi^{(k)}$ for all $k$
large enough, after which one deletes the following ones in $\zeta^{(k)}$.
\endproof

\noindent Given $k\in\NN$ and $\psi\in \Psi$, let
\begin{equation}
i_k(\psi)=
\begin{cases}
\min\{n\in\NN\colon \; \psi_n\geq k \}, &\text{if such $n< + \infty$ exists,}
\\
+\infty, \quad  &\text{if $\psi_n < k$ for all $n \ge 1$}.
\end{cases}
\end{equation}

\begin{definicao} \label{Mspaced}
Let $M\ge 2$ be an integer. A sequence $\psi\in \Psi$ is called
$M$-spaced up to level $k$ if the following conditions are
satisfied:
\begin{align}
&a) & j-i&\geq M^{\min\{\psi_i,\psi_j\}}, \; \text{ for all }\,1
\leq i<j,\text {  both in } \NN, \label{10k}
\\
&b_k) &  i_j(\psi)&\geq M^{j},\;\text{ for all  } \,j \le k. \label{9k}
\\
&\, & \, & \notag
\\
\!\!\!\! \, \text {A sequence }&\psi\in
\Psi & &\!\!\!\!\!\!\!\!\!\!\!\!\!\!\!\!\!\mbox {is called M-spaced when a) holds and $b_k$) is replaced by $b^{\prime}$): } \notag
\\
&\, & \, & \notag
\\
&b^{\prime}) & i_j(\psi)&\geq M^{j},\;\text{ for all  } \,j \geq 1.
\label{9}
\end{align}
\end{definicao}
\noindent We denote
\begin{align}
\Psi_M^{k} & :=\{\xi\in \Psi \colon \xi\text{ is $M$-spaced up to level $k$}\}, \notag
\\
\Psi_M & :=\{\xi\in \Psi \colon \xi\text{ is $M$-spaced}\}. \notag
\end{align}
\begin{definicao} \label{zeta}
Let $L\geq 2$ be an integer. Let $\zeta{(L)} \in \Psi_L$ be the
sequence whose $j$-th entry, $j \ge 1$, is
given by:
\begin{equation}
\label{hierarquica} (\zeta{(L)})_j :=\begin{cases} k,\; &\mbox{ if }\;
L^k|j\ \mbox{ but }\ L^{k+1}\nmid j\,, \\ 0, &\mbox{ if } \; L\nmid
j.
\end{cases}
\end{equation}
\end{definicao}

We now state the main result of this section.

\begin{teorema}
\label{weight} Let $L\geq 2$ and  $M\ge 3 (L+1)$ be integers,
$\psi \in \Psi_M$, and $\zeta{(L)} $
given by (\ref{hierarquica}). Then the configuration $\omega_{\zeta{(L)},\psi}$ defined in \eqref{25aug5}
has an infinite open permitted path $\pi$ starting from the origin.
\end{teorema}


\noindent The proof of Theorem \ref{weight} will follow from a sequence of technical
statements, the most important is given in Proposition \ref{c1}, whose proof, in turn, relies
on that of Proposition \ref{st}.

\medskip

\begin{remark}
\label{finito} For the proof of Theorem \ref{weight} (see below) it
is enough to treat the case when $i_k(\psi)<+\infty$ for all
$k\ge 1$. For this reason, in the next statements we shall assume
that all featuring $i_k(\psi)$ are finite.
\end{remark}

\medskip

\noindent For every $m \in \NN$ and $\xi\in \Psi$, let $\theta_m \xi
\in \Psi$ denote the shifted sequence given by $(\theta_m \xi)_j :=
\xi_{j+m}, \; j \geq 1$. Set $\theta_0 \xi=\xi$. From Definition
\ref{zeta} it follows that for every  $m$ and $n \in \NN$, the
sequence $\theta_{mL^n} \zeta{(L)}$ satisfies, for all $j \ge 1$:
\begin{equation}
(\theta_{mL^n} \zeta{(L)})_{j} \;
\begin{cases}
= 0 \; &\mbox{ if }   L\nmid j,
\\
= n' \;   &\mbox{ if }  L^{n'}\mid j \; \text{but} \; L^{n'+1}\nmid
j, \mbox{ for }    1 \le n' < n,
\\
\ge n & \mbox{ if }  L^{n}\mid j.
\end{cases} \label{A1}
\end{equation}

\noindent {\bf{Remark.}} The sequence  $\theta_{L^k} \zeta{(L)}$ is
not  in $\Psi_L$, since the inequality (\ref{9}) is violated.
However $\theta_{mL^k} \zeta{(L)} \in \Psi_L^{k}$ for any $m \in
\mathbb{N}$.

\begin{proposition} \label{p1} If
$\psi \in \Psi_M^k$ for some $k \in \mathbb{N}$, and $i \ge 1$ is
such that $\psi_i = m \ge 1 $, then $\theta_i \psi \in \Psi_M^m$. In
particular, if $\psi \in \Psi_M$, then $\theta_{i_k(\psi)} \psi \in
\Psi_M^{\psi_{i_k(\psi)}}$.
\end{proposition}

\begin{proof} To verify inequality (\ref{10k}) for $\theta_i \psi$ observe that since $\psi \in \Psi_M^k$ and if $j > j^\prime$, then
$$
M^{\min\{(\theta_i \psi)_j, (\theta_i \psi)_{j^\prime}\}} = M^{\min\{\psi_{i+j},\psi_{i+j^\prime}\}} \leq j - j^\prime.
$$
On the other hand, if $j \ge 1$ is such that $(\theta_i \psi)_j \leq
m$, then  we have:
$$
M^{(\theta_i \psi)_j} = M^{\psi_{j+i}} = M^{\min\{\psi_{i+j},\psi_{i}\}} \leq i+j -i = j,
$$
and thus $\theta_i \psi$ satisfies (\ref{9k}).
\end{proof}

\noindent {\bf{Notation.}} Since the parameter $L \geq 2$ is fixed, we will omit it from the notation when this will
bring no confusion.

\medskip
\noindent Define the horizontal slab $\mathbb{S}_{[i,j]} := \{v = (v_1, v_2) \in \mathbb{Z}_+
\colon\, i \le v_2 \le j  \}$.

\begin{lema}\label{cl33}
Let $\zeta=\zeta{(L)}$ be as in Definition \ref{zeta}. For any $k\in
\ZZ_+$,  $m \in \mathbb{N}$ and $\psi \in \Psi$, if $i_k=i_k(\psi)$,
\begin{equation} \label{A2}
{\mathcal{C}}_{\langle \theta_{mL^{\psi_{i_k}}}\zeta,\psi \rangle}\cap
\mathbb{S}_{[0,i_k]} = {\mathcal{C}}{_{\langle\zeta,\psi\rangle}}\cap
\mathbb{S}_{[0,i_k]},
\end{equation}
in particular, for any $m$, $m' \in \mathbb{N}$
\begin{equation} \label{AA2}
{\mathcal{C}}_{\langle\theta_{mL^{\psi_{i_k}}}\zeta,\psi\rangle}\cap
\mathbb{S}_{[0,i_k]} = {\mathcal{C}}_{\langle\theta_{m'L^{\psi_{i_k}}} \zeta ,\psi\rangle}\cap
\mathbb{S}_{[0,i_k]}.
\end{equation}
\end{lema}
\begin{proof}
Since the origin is the unique open vertex along the axes, to prove
\eqref{A2} it suffices to consider vertices $ v = (v_1, v_2)$ with
$1\le v_1,1\le v_2 \le i_k$, and to show that $\zeta_{v_1} \ge
\psi_{v_2}$ if and only if $\zeta_{v_1+mL^{\psi_{i_k}}} \ge
\psi_{v_2}$.

Let us first assume that $\zeta_{v_1}  \geq \psi_{v_2}$. By the
first two equalities of (\ref{A1}) we have that $\zeta_{v_1 +
mL^{\psi_{i_k}}}=\zeta_{v_1}$ whenever $v_1$ is not multiple of
$L^{\psi_{i_k}}$. If $v_1$ is multiple or $L^{\psi_{i_k}}$, then by
the third inequality of (\ref{A1}) we have that $\zeta_{v_1 + m
L^{\psi_{i_k}}} \ge \psi_{i_k}$. Now we use the fact that $v_2 \leq
i_k$, which implies that $\psi_{v_2} \le \psi_{i_k}$,  due to the
definition of $i_k$. Thus we have $\zeta_{v_1 + m L^{\psi_{i_k}}}
\ge \psi_{v_2}$ in both cases.

On the other hand, if  $\zeta_{v_1} < \psi_{v_2}$  and $v_2 \le
i_k$, we get $\zeta_{v_1} < \psi_{v_2} \le \psi_{i_k}$. This implies
that $v_1$ cannot be multiple of $L^{\psi_{i_k}}$. Thus, as just
observed, $\zeta_{v_1 + mL^{\psi_{i_k}}}  = \zeta_{v_1}  < \psi_{v_2} $.

As already explained, this proves \eqref{A2}. Equality \eqref{AA2}
follows immediately, concluding the proof of Lemma \ref{cl33}.
\end{proof}

\medskip

\noindent The set $I$ is called {\it horizontal discrete segment} if
the elements of the set $I$ are vertices of $\ZZ^2_+$ whose first
coordinates are consecutive integers and the second coordinate is
the same for all elements. We denote by $l(I)$ and $r(I)$
respectively the value of the abscissa coordinate of the left and
right endpoints of $I$.

\medskip

\noindent For $\zeta,\psi \in \Psi$ we set:
\begin{equation}
\label{permitido} V_{\langle \zeta^{},\psi \rangle } (k)  :=  \{ v
\in \ZZ_+\times\{k\}:
\exists \text{ open permitted path } \pi \text{ from the origin to } v \}
\end{equation}

\begin{remark} \label{trivial1}
Let $M\ge 2$ and $\psi \in \Psi_M^1$. Since $i_1(\psi)\ge M\ge 2$ and
$\psi_i=0$ for $i<i_i(\psi)$, all the vertices in
$\{(v_1,v_2)\colon 1\le v_1, v_2\le i_1(\psi)-1\}$ are open in the configuration
$\omega_{\zeta,\psi}$, for any $\zeta\in \Psi$. It easily follows that
\begin{equation}
\label{trivial}
 V_{\langle \zeta,\psi \rangle } (i_1 (\psi)-1) = \{(j,i_1(\psi)-1);\
j=1,\dots ,i_1(\psi)-1\}.
\end{equation}
for all $\psi \in \Psi_M^1, \zeta \in \Psi$. (Indeed, any open
oriented path connecting the origin to the set $\{(j,i_1(\psi)-1);\ j=1,\dots ,i_1(\psi)-1\}$
will also be permitted.) \end{remark}

The next proposition summarizes some basic properties of
$V_{\langle \zeta^{},\psi \rangle}(k)$ in the case of interest for
Theorem \ref{weight}.

\noindent {\bf Remark.} Throughout we adopt the usual convention $\sum_{r=1}^0 a_r=0$ for $a_r\in \mathbb{R}$.)

\medskip

\begin{proposition} \label{c1} Let $\zeta=\zeta{(L)}$ be as in Definition \ref{zeta} and $\psi \in \Psi_M^k$,
with $M \ge 3(L+1)$ and $k \ge 1$. For any $m \in \ZZ_+$ the
following holds:
\begin{align}
&{\it i)} &  &  V_{\langle \theta_{mL^{k}}\zeta^{},\psi \rangle } (i_k (\psi)-1)  \; \; \mbox{  is a
discrete segment. } \notag
\\
& {\it ii)} & & l\big(V_{\langle \theta_{mL^{k}}\zeta^{},\psi
\rangle } (i_k (\psi)-1)\big ) \leq \max\left\{\sum_{r=1}^{k-1}(L^r
+1) \big \lfloor \frac{i_k(\psi)-1}{M^r} \big \rfloor, 1\right\} .
\notag
\\
& {\it iii)} && r\big(V_{\langle \theta_{mL^{k}}\zeta^{},\psi
\rangle} (i_k (\psi)-1)\big) \ge i_k(\psi)-
\max\left\{\sum_{r=1}^{k-1}(L^r +1) \big \lfloor
\frac{i_k(\psi)-1}{M^r} \big \rfloor, 1\right\}. \notag
\end{align}

\end{proposition}
\medskip
\noindent {\it Proof of Proposition \ref{c1} (beginning)}

We shall proceed by induction. The statements are trivial in the case $k=1$, where indeed
equality holds, cf. Remark \ref{trivial1}.

\medskip




\noindent By induction hypothesis assume now that
$V_{\langle \theta_{mL^{k}}\zeta^{},\psi \rangle} (i_k (\psi)-1)$  satisfies $i), \; ii),$ and
$iii)$ for $k$, and we proceed to the case $k+1$. For that we assume that $\psi \in \Psi^{k+1}_M$.
\medskip

\noindent If $i_{k+1}(\psi) = i_k(\psi)$, then $V_{\theta_{mL^{k+1}}\zeta^{},\psi} (i_{k+1} (\psi)-1) =
 V_{\theta_{mL^{k+1}}\zeta^{},\psi} (i_k (\psi)-1)$ and, therefore, $i), \; ii),$ and $iii)$ are satisfied.
 We thus consider the case $i_{k+1}(\psi) > i_k(\psi)$. The last inequality implies that $\psi_{i_k(\psi)}=k$.



\noindent To proceed forward we will need the following proposition.
\begin{proposition} \label{st} Let $\zeta$ and $M$ be as in the statement of Proposition \ref{c1}.  Assume that
{\it i) - iii)} of Proposition \ref{c1} hold for some $k\ge 1$ and that
$\psi$ satisfies property (\ref{10k}) with $M \ge 3(L+1)$. Let the
index $ j_1$ be such that $\psi_{j_1}=k$. Define $j_2 := j_1 + i_k
(\theta_{j_1} \psi)$, {\it i.e.} $j_2$ is the smallest index $i >
j_1$, such that $\psi_i \ge k$. Fix a discrete segment $I_{1}
\subset \NN \times \{j_1-1\}$, such that $|I_{1}| \geq L^k$ (we use
$|I|$ to denote the cardinality of a set $I$). Then the following
holds:
The set
\begin{equation*}
 I_2  := \{ v \in \NN\times \{ j_2-1\}:
\exists \mbox{ open permited path $\pi$ from $I_1$ to $v$} \} \qquad   \;  \; \notag
\end{equation*}
is a discrete segment
and
\begin{align} \label{A34}
l ({I_2}) & \leq
l ({I_1}) + L^k + \sum_{r=1}^{k-1}(L^r +1) \Big \lfloor\frac{j_2 -j_1 -1}{M^r}\Big\rfloor,
\\
\label{A35}
r ({I_2}) & \geq
r({I_1}) + (j_2 - j_1-1) -  L^k - \sum_{r=1}^{k-1}(L^r +1) \Big \lfloor\frac{j_2 - j_1 -1}{M^r}\Big\rfloor.
\end{align}
In particular we have
\begin{equation} \label{A6}
|I_{2} | \geq
| I_{1} |
 + (j_2 - j_1-1) - 2L^k - 2\sum_{r=1}^{k-1}(L^r +1) \Big \lfloor\frac{j_2 - j_1 -1}{M^r}\Big\rfloor.
\end{equation}
\end{proposition}
\medskip

\noindent {\it Proof of Proposition \ref{st}}. Define
$$
U_{1}= \{u=(u_1, j_1):  (u_1, j_1-1)\in I_{1} \mbox{ and }  L^k \mid u_1\}.
$$
Since $|I_{1}| \geq L^k$  we have
\begin{equation}
|U_{1} | \geq
\begin{cases}
1, &\mbox{ if } L^k \leq |I_{1}| < 3L^k,
\\
\Big\lfloor \frac{|I_{1}| - 2L^{k}}{L^k} \Big\rfloor &\mbox{ if } |I_{1}| \geq 3L^k.
\end{cases}
\end{equation}
Enumerate the vertices of $U_{1}$ in increasing first coordinate order:
$$
U_{1}= \{u^1 = (u^1_1, j_1), \dots, u^{s_1} = (u^{s_1}_1, j_1) \}.
$$
Since $\zeta_{u^r_1}  \geq k$, for all $1 \leq r \leq s_1$, under assumption that $\psi_{j_1}=k$  we have:
$$
{u}^r = (u^r_1, j_1) \in {\mathcal{C}}_{\langle \zeta,\psi \rangle } (I_{1}) \cap (\ZZ\times \{j_1\}), \quad r = 1, \dots, s_1.
$$
The same argument as in the proof of Lemma \ref{cl33} applied to pair of sequences $\zeta$ and $\theta_{j_1} \psi$,  gives
\begin{equation}
{\mathcal{C}}_{\langle \theta_{u_1^1} \zeta, \, \theta_{j_1} \psi \rangle} \cap \mathbb{S}_{[0, \, j_2-j_1-1]}=
{\mathcal{C}}_{\langle \theta_{u_1^r}\zeta, \, \theta_{j_1} \psi \rangle}  \cap \mathbb{S}_{[0, \, j_2-j_1-1]}, \;
\mbox{ for }  2 \leq r \leq s_1,
\end{equation}
{\it i.e.} each ${\mathcal{C}}_{\langle \zeta,\psi \rangle } (u_1^r)  \cap \mathbb{S}_{[j_1, j_2-1]}$ is shift
of ${\mathcal{C}}_{\langle \zeta,\psi \rangle } (u_1^1) \cap \mathbb{S}_{[j_1, j_2-1]}$ by ${(r-1)L^k}$ to the right.

\medskip

\noindent By Proposition \ref{p1}, the sequence $\theta_{j_1}\psi$ satisfies the conditions of Proposition \ref{c1}
(recall that $j_1$ is such that $\psi_{j_1} = k$). Thus, applying {\it i) - iii)} subsequently to pairs of sequences
$\theta_{u_1^r}\zeta$ and $\theta_{j_1} \psi$, for $1 \leq r \leq s_1$, we have, thanks to validity of {\it i)}
for $k$, that the set
$$
V_{\langle \theta_{u_1^r}\zeta, \, \theta_{j_1} \psi \rangle} (i_k (\theta_{j_1} \psi)-1)
$$
is a discrete segment for all $1 \leq r \leq s_1$. Next we argue that
\begin{equation} \label{concat}
I_2 =
\bigcup_{r=1}^{s_1} [V_{\langle \theta_{u_1^r}\zeta, \, \theta_{j_1} \psi \rangle} (i_k (\theta_{j_1} \psi)-1) + u^r].
\end{equation}
Though the validity of the equality (\ref{concat}) is nearly obvious,  it is not immediate. One should take care of the following: if there are two permitted paths $\pi$ and  $\pi^\prime$, from $x$ to $y$ and  from $y$ to $z$, respectively,
their concatenation at the vertex $y$ does not necessarily result in a permitted path from $x$ to $z$. However  (recall Definitions \ref{percolacao} and \ref{permit}) when $y$ is a heavy vertex, as the case here for  $y=u^r$, one automatically gets a permitted path.


\medskip

\noindent Next we shall show that for all $2 \leq r \leq s_1 $
\begin{equation} \label{lr}
r\big(V_{\langle \theta_{u_1^{r-1}}\zeta, \, \theta_{j_1} \psi \rangle}  (i_k (\theta_{j_1} \psi)-1) + u^{r-1}\big)
\geq
l\big(V_{\langle \theta_{u_1^r}\zeta, \, \theta_{j_1} \psi \rangle} (i_k (\theta_{j_1} \psi)-1) + u^r \big)
\end{equation}
For this we need the following
\begin{lema}\label{buraco}Let $\tilde L$ and $M\geq 3\tilde L$ be positive integers.
Then for any $ k \in \NN$ and $a\geq M^k$, we have:
\begin{equation} \label{bu}
a\geq \tilde L^k + 2a \sum_{j=1}^{k-1} (\tilde L/M)^j.
\end{equation}
\end{lema}

\begin{proof}
A trivial computation shows that the r.h.s. of \eqref{bu} is bounded
from above by $\tilde L^k + a (1-3^{-k+1})$, and the inequality follows at
once.
\end{proof}


\medskip
\noindent We return to the proof of (\ref{lr}). Applying {\it ii)} and {\it iii)} we have
\begin{align}
r\big(V_{\langle \theta_{u_1^{r-1}}\zeta, \, \theta_{j_1} \psi \rangle} (i_k (\theta_{j_1} \psi)-1) + & u^{r-1}\big) \notag
\\
\geq \;
 & j_2 - j_1 -1-\sum_{i=1}^{k-1}(L^i + 1) \Big \lfloor \frac{j_2 - j_1 - 1}{M^i} \Big \rfloor + u_1^{r-1} \notag
\\
\geq \; &   L^k + \sum_{i=1}^{k-1}(L^i + 1) \Big \lfloor \frac{j_2 - j_1 - 1}{M^i} \Big \rfloor + u_1^{r-1}   \notag
\\
\geq \; &   \sum_{i=1}^{k-1}(L^i + 1) \Big \lfloor \frac{j_2 - j_1 - 1}{M^i} \Big \rfloor + u_1^{r} \notag
\\
\geq \;  & l(V_{\langle \theta_{u_1^r}\zeta, \, \theta_{j_1} \psi \rangle} (i_k (\theta_{j_1} \psi)-1) + u^r), \label{overlap}
\end{align}
for all $2 \leq r \leq s_1$, and where the second inequality is
implied by (\ref{bu}) with $\tilde L=L+1$. Inequality
(\ref{overlap}) immediately implies that $I_2$ is the discrete
segment. Next we will show (\ref{A34}) and (\ref{A35}):

\begin{align}
l(I_2) = \; & l(V_{\langle \theta_{u_1^1}\zeta, \, \theta_{j_1} \psi \rangle} (i_k (\theta_{j_1} \psi)-1) + u^1 \notag
\\
\leq \; & l(V_{\langle \theta_{u_1^1}\zeta, \, \theta_{j_1} \psi \rangle} (i_k (\theta_{j_1} \psi)-1) + l(I_1) + L^k \notag
\\
\leq \; & l(I_1) + L^k +  \sum_{i=1}^{k-1}(L^i + 1) \Big \lfloor \frac{i_k (\theta_{j_1} \psi)-1}{M^i} \Big \rfloor \notag
\\
= \; & l(I_1) + L^k +  \sum_{i=1}^{k-1}(L^i + 1) \Big \lfloor \frac{j_2 - j_1 -1}{M^i} \Big \rfloor,  \notag
\end{align}
where in the third line inequality of the above display we used {\it ii)}. This proves (\ref{A34}). On the other hand
\begin{align}
r(I_2)=  \; & r\big(V_{\langle \theta_{u_1^{s_1}}\zeta, \, \theta_{j_1} \psi \rangle} (i_k (\theta_{j_1} \psi)-1) \big) + u^{s_1} \notag
\\
\geq \; & r\big(V_{\langle \theta_{u_1^{s_1}}\zeta, \, \theta_{j_1} \psi \rangle} (i_k (\theta_{j_1} \psi)-1)\big) +  r(I_1) -L^k \notag
\\
\geq \; & r(I_1)  - L^k + i_k (\theta_{j_1} \psi)-1 - \sum_{i=1}^{k-1}(L^i + 1) \Big \lfloor \frac{i_k (\theta_{j_1} \psi)-1}{M^i} \Big \rfloor \notag
\\
= \; & r(I_1) + (j_2 - j_1 -1) - L^k - \sum_{i=1}^{k-1}(L^i + 1) \Big \lfloor \frac{j_2 - j_1 -1}{M^i} \Big \rfloor,  \notag
\end{align}
where in the third line inequality of the above display we used {\it ii)}. This proves (\ref{A35}) and finishes the proof of the Proposition \ref{st}. $\Box$

\bigskip

\begin{remark} \label{r1} \rm{ Proposition \ref{st}
remains valid if we replace $\zeta$ by any shift $\theta_n \zeta$ of
it, in particular by $\theta_{mL^k} \zeta$.}
\end{remark}
\medskip

\noindent \noindent {\it Proof of Proposition \ref{c1}  (continuation)}

Now we return to the proof of {\it i) - iii)} for the case $k+1$, in order to complete the induction step. Recall
that we assume that $i_{k+1} (\psi) > i_{k} (\psi) $, which implies
that $\psi_{i_{k} (\psi)} = k$, and  we also assume that $\psi \in \Psi_M^{k+1}$.

\medskip

\noindent Define iteratively:
\begin{align}
j_1 : =\; & i_k (\psi), \notag
\\
j_n : =\; &  j_{n-1} + i_k (\theta_{j_{n-1}}\psi), \quad \mbox{ for } \quad n=2, \dots, n^*, \notag
\end{align}
where $n^*=\min\{n\colon \psi_{j_n}\ge k+1\}$ that is, $j_{n^*}=i_{k+1}(\psi)$. Fix $m \geq 0$, and define
\begin{align}
I_1 : =\; & V_{\langle \theta_{mL^{k+1}}\zeta^{},\psi \rangle } (i_k (\psi)-1), \notag
\\
I_n : =\; &  \{ x \in {\mathcal{C}}_{\langle \theta_{mL^{k+1}} \zeta^{},\psi \rangle} (I_{n-1}) \cap (\ZZ\times \{ j_n-1\}):
\exists \mbox{ permited path from $I_{n-1}$ to $x$} \} \notag
\end{align}
for $2 \leq n \leq n^*$.

\bigskip

\noindent First of all we notice that, as in the proof of Proposition \ref{st}:
$$
I_{n^*} = V_{\langle \theta_{mL^{k+1}}\zeta^{},\psi \rangle } (i_{k+1} (\psi)-1).
$$

\medskip

\noindent Observe now that the sequence $\theta_{j_{n-1}}\psi$, for any $1 \leq n \leq n^*$ satisfies the conditions of Proposition \ref{st}. We will
apply it to each $I_n$ iteratively, starting from $I_1$.
\medskip

\noindent {\it i) - case k+1.} Due to the first part of the Proposition \ref{st} we have that all $I_n$, $1 \leq n \leq n^*$ are discrete segments,
and therefore $V_{\langle \theta_{mL^{k+1}}\zeta^{},\psi \rangle } (i_{k+1} (\psi)-1)$ is a discrete segment.

\noindent {\it ii) - case k+1.}
\begin{align}
l (I_{n^*}) \equiv \; & l\big(V_{\langle \theta_{mL^{k+1}}\zeta^{},\psi \rangle } (i_{k+1} (\psi)-1) \big)  \notag
\\
\leq \; & l (I_{n^*-1}) + L^k +  \sum_{r=1}^{k-1}(L^r + 1) \Big \lfloor \frac{j_{n^*} - j_{n^*-1} -1}{M^r} \Big \rfloor  \notag
\\
\dots \; & \notag
\\
\leq \; & l (I_{1}) + (n^* -1 )L^k +  \sum_{s=2}^{n^*}\sum_{r=1}^{k-1}(L^r + 1) \Big \lfloor \frac{j_{s} - j_{s-1} -1}{M^r} \Big \rfloor  \notag
\\
\leq \; &  n^* L^k +  \sum_{s=1}^{n^*}\sum_{r=1}^{k-1}(L^r + 1) \Big \lfloor \frac{j_{s} - j_{s-1} -1}{M^r} \Big \rfloor  \notag
\\
\leq \; & \Big \lfloor\frac{i_{k+1}(\psi)-1}{M^k}\Big \rfloor L^k + \sum_{r=1}^{k-1} (L^r + 1) \sum_{s=1}^{n^*}
\Big \lfloor \frac{j_{s} - j_{s-1} -1}{M^r} \Big \rfloor \notag
\\
\leq \; & \Big \lfloor\frac{i_{k+1}(\psi)-1}{M^k}\Big \rfloor L^k + \sum_{r=1}^{k-1} (L^r + 1)
\Big \lfloor \frac{j_{n^*} -1}{M^r} \Big \rfloor \notag
\\
\leq \; &  \sum_{r=1}^{k} (L^r + 1)
\Big \lfloor \frac{i_{k+1}(\psi) -1}{M^r} \Big \rfloor, \notag
\end{align}
which proves validity of {\it ii) } for the case $k+1$.
\medskip

\noindent {\it iii) - case k+1.}
\begin{align}
r (I_{n^*}) \equiv \; & r\big(V_{\langle \theta_{mL^{k+1}}\zeta^{},\psi \rangle } (i_{k+1} (\psi)-1) \big)  \notag
\\
\geq \; &
r (I_{n^*-1}) + ({j_{n^*} - j_{n^*-1} -1}) -  L^k - \sum_{r=1}^{k-1}(L^r + 1) \Big \lfloor\frac{{j_{n^*} - j_{n^*-1} -1}}{M^r}\Big\rfloor \notag
\\
\dots \; & \notag
\\
\geq \; &  r (I_{1}) + \sum_{s=2}^{n^*}({j_{s} - j_{s-1} -1}) - n^* L^k - \sum_{s=2}^{n^*}\sum_{r=1}^{k-1}(L^r + 1) \Big \lfloor\frac{{j_{s} - j_{s-1} -1}}{M^r}\Big\rfloor \notag
\\
\geq \; &   \sum_{s=1}^{n^*}({j_{s} - j_{s-1} -1}) - n^* L^k - \sum_{s=1}^{n^*}\sum_{r=1}^{k-1}(L^r + 1) \Big \lfloor\frac{{j_{s} - j_{s-1} -1}}{M^r}\Big\rfloor \notag
\\
\geq \; & j_{n^*} - n^* -n^*L^k - \sum_{r=1}^{k-1}(L^r + 1) \Big \lfloor\frac{{j_{n^*} -1}}{M^r}\Big\rfloor \notag
\\
\geq \; & i_{k+1}(\psi) - \sum_{r=1}^{k}(L^r+1) \Big \lfloor\frac{{i_{k+1}(\psi) -1}}{M^r}\Big\rfloor, \notag
\end{align}
which proves validity of {\it iii) } for the case $k+1$, and finishes proof of the Proposition \ref{c1}. $\Box$


\bigskip

\noindent {\it Proof of Theorem \ref{weight}.}

If $\psi \in \Psi$ is such that $i_k(\psi)<\infty$ for all $k$, the
statement follows at once from Proposition \ref{c1}. If $i_{k(1)}(\psi)
<+\infty$ and  $i_{k(1)+1}(\psi) =+\infty$, one repeats the argument
of Proposition \ref{c1} to the sequence $\theta_{i_{k(1)}}(\psi)$ for
$j_1,j_2,...$ as in the proposition with $k=k(1)$; if this stops,
then one moves down to the next $k(2)<k(1)$, and so on.\qed

\medskip

\noindent The next result is a straightforward corollary of Theorem
\ref{weight} and Lemma \ref{compatibility-lemma}.

\begin{corolario}\label{pesos} If $\zeta{(L)}$ is given by \eqref{hierarquica} and $\psi \in \Psi_{M}$
with $M=3(L+1)$, then the pair $(\zeta{(L)},\psi)$ is compatible.
\end{corolario}


\section{Grouping \label{III}}

In this Section we show that if a binary sequence $\xi$ is sampled from ${\mathbb P}_p$ with low density $p$ of ones,
then with positive probability it can be viewed in a certain sense as an ``$M$-spaced sequence".
The precise statement is formulated in Corollary \ref{sec}. This is obtained via an algorithm that suitably groups the ones in $\xi$ into
{\it clusters} and attributes an adequate weight or {\it mass} to each cluster. The construction, with a hierarchical structure, was
developed in \cite{KSV} and is presented here for sake of completeness. This section is mainly devoted to the description of this grouping
procedure. The proofs of its convergence and further consequences needed here are taken from \cite {KSV} and included in the Appendix.

Let $\xi\in\Xi$ be distributed according to $\mathbb{P}_p$, where $p=\mathbb P_p(\xi_i=1)$ will be assumed small.  Let
$\Gamma\equiv\Gamma (\xi)=\{i \in \NN\colon \xi_i=1\}$. We shall decompose $\Gamma$ into sets $\mathcal{C}_{i}$, called {\it clusters},
to which an $\NN$-valued mass $m(\mathcal{C}_i)$ is attributed ($m(\mathcal{C}_i) \le |\mathcal{C}_i|$) in a way that
$d(\mathcal{C}_i,\mathcal{C}_j)\geq M^{\min\{m(\mathcal{C}_i),m(\mathcal{C}_j)\}}$, where $d(D_1,D_2)$ denotes the usual Euclidean distance
between two sets $D_1$ and $D_2$, and $|C|$ stands for the cardinality of the set $C$.

The clusters $\mathcal{C}_i:=\mathcal{C}_{\infty,i}$  will be obtained by a limiting recursive procedure. We will build
an infinite sequence $\{{\text{\bf C}}_k\}_{k\ge 0}$ of partitions of $\Gamma$. Each partition ${\text{\bf C}}_k$ is a
collection  ${\text{\bf C}}_k = \{ {{\mathcal{C}}}_{k,j} \}_{j\ge 1}$ of
subsets of $\Gamma$. The construction depends on the parameter $M$ ($M\ge 2$ a large integer to be fixed later according to the
conditions of Lemma \ref{lemma1.1} below). The clusters will be constructed as to have the properties
\begin{equation}
\label{2.z}
\text{each ${{\mathcal{C}}}_{k,j}$ is of the form $I\cap \Gamma$ for an
interval $I$},
\end{equation}
and
\begin{equation}
\label{2.one}
\text{span}({{\mathcal{C}}}_{k,j}) \cap \text{span}({{\mathcal{C}}}_{k,j'}) =
\emptyset \quad \text {if }\; j \neq j',
\end{equation}
where $span(C)$ is the smallest interval (in $\mathbb
Z_+$) that contains $C$. To each cluster ${\mathcal{C}}_{k,j}$  we will
attribute a {\it mass}, $m({\mathcal{C}}_{k,j})$, in such a way that
\begin{equation}
\label{2.two}
d({\mathcal{C}}_{k,j}, {\mathcal{C}}_{k,j'}) \ge M^r, \; \; \text{if} \; \min
\{m({\mathcal{C}}_{k,j}), m ({\mathcal{C}}_{k,j'})  \} \ge r, \; \; \text{for}
\; \; r = 1, \dots, k \text{ and } j \ne j'.
\end{equation}
To each cluster ${\mathcal{C}}_{k,j}$
we shall further associate a number $\ell({\mathcal{C}}_{k,j}) \in
\{0,1,\dots,k\}$ which will be called the {\it level of the
cluster}. This level will satisfy
\begin{equation}
\label{lev}
0 \le \ell ({\mathcal{C}}_{k,j}) < m({\mathcal{C}}_{k,j}).
\end{equation}

\noindent Finally we construct the (limiting) partition ${\text{\bf C}}_{\infty}= \{ {\mathcal{C}}_{\infty ,j} \}_{j\ge 1}$ of $\Gamma$:
\begin{equation}
\label{2.onea}
\text{span}({\mathcal{C}}_{\infty,j}) \cap \text{span}({\mathcal{C}}_{\infty,j'}) = \emptyset \quad \text {if} \; j \neq j'.
\end{equation}
To each cluster ${\mathcal{C}}_{\infty,j}$ of ${\text{\bf C}}_{\infty}$
we will attribute a mass $m ({\mathcal{C}}_{\infty,j})$, and a level
$\ell({\mathcal{C}}_{\infty,j})$, in such a way that
\begin{equation}
\label{2.a}
0 \le \ell({\mathcal{C}}_{\infty,j}) < m ({\mathcal{C}}_{\infty,j}),
\end{equation}
and the following property holds:
\begin{equation}
\label{2.four}
d({\mathcal{C}}_{\infty,j}, {\mathcal{C}}_{\infty,j'}) \ge M^r, \; \; \text{if}
\; \min \{m({\mathcal{C}}_{\infty,j}), m({\mathcal{C}}_{\infty,j'})  \} \ge r,
\; \; \text{for} \; \; r \ge 1 \text{ and } j \ne j',
\end{equation}
or equivalently,
\begin{equation}
\label{2.five}
d({\mathcal{C}}_{\infty,j}, {\mathcal{C}}_{\infty,j'}) \ge M^{\min \{m({\mathcal{C}}_{\infty,j}),
m({\mathcal{C}}_{\infty,j'})  \}}, \text {for} \; j \neq
j'.
\end{equation}

\noindent {$\mathbf{The \; construction.}$}
\medskip
Let the elements of  $\Gamma$ be labeled in increasing
order: $\Gamma = \{ x_j\}_{j\ge 1} $ with $x_1 < x_2 < \dots$.

\medskip
\noindent {\bf Level 0.} The clusters of level 0 are just the
subsets of $\Gamma$ of cardinality one. We take ${\mathcal{C}}_{0,j} =
\{x_j\}$ and attribute a unit mass to each such cluster. That is,
$m ({\mathcal{C}}_{0,j})=1$ and $\ell({\mathcal{C}}_{0,j})=0$. Set  ${\text{\bf
C}}_{0,0}={\text{\bf C}}_0 = \{ {\mathcal{C}}_{0,j} \}_{j\ge 1}$. Further
define $\alpha({\mathcal{C}}) = \omega({\mathcal{C}}) = x$  when ${\mathcal{C}} = \{x\}$ is a
cluster of level 0.

\medskip

\noindent {\bf  Level 1.} We say that $x_i, x_{ i+1}, \dots,
x_{i+n-1}$ form  a {\it maximal $1$-run of length $n\ge 2$} if
$$
x_{j+1}-x_j\, < \, M,  \; j=i, \dots ,i+n-2,
$$
and
$$
x_{j+1}-x_j\, \ge \, M  \; \;
\begin{cases} \text{for} \; \; j=i-1, j= i+n-1, \; &\text{if } \; i>1 \\
\text{for} \; \; j= i+n-1, \; &\text{if } \; i=1. \end{cases}
$$
The level 0 clusters $\{x_i\}, \{x_{i+1}\}, \dots, \{x_{i+n-1}\}$
will be called {\it constituents} of the run. Note that there are
no points in $\Gamma$ between two consecutive points of a maximal
1-run. Also note that if $x_{j+1}-x_j \ge M$ and $x_j-x_{j-1} \ge
M$, then $x_j$ does not appear in any maximal 1-run of length at
least 2.

For any pair of distinct maximal runs, $r'$ and $r''$ say, all
clusters in $r'$ lie to the left of all clusters in $r''$ or vice
versa. It therefore makes sense to
label the  consecutive maximal $1$-runs of length at least $2$ in
increasing order of appearance: $r^1_1, r^1_2, \dots $. It is
immediate that $\mathbb{P}_p$-a.s. all runs are finite, and that infinitely
many such runs exist. We write $r^1_i = r^1_i(x_{s_i}, x_{s_i +1},
\dots, x_{s_i+n_i-1}), \; i=1, \dots$, if the $i$-th run consists
of $x_{s_i}, x_{s_i +1}, \dots, x_{s_i+n_i-1}$. Note that $n_i\ge
2$ and $s_i+n_i \le  s_{i+1}$ for each $i$. The set
$$
  {\mathcal{C}}^1_{i} =\{x_{s_i},x_{s_i+1},\dots, x_{s_i+n_i-1}\}
$$
is called a {\it level 1-cluster}, i.e., $\ell( {\mathcal{C}}^1_i)=1$.
We attribute to ${\mathcal{C}}^1_i$ the mass given by
its cardinality:
$$
m ({\mathcal{C}}^1_i) = n_i.
$$
The points
$$
\alpha^1_i = x_{s_i}\quad \text{and}\quad \omega^1_i =
x_{s_i+n_i-1}
$$
are called, respectively, the {\it start-point} and {\it
end-point} of the run, as well as of the cluster ${\mathcal{C}}^1_i$.
To avoid confusion we sometimes  write more explicitly
$\alpha({\mathcal{C}}^1_i), \;\omega({\mathcal{C}}^1_i)$.

 By ${\text{\bf C}}_{1,1}$ we denote the set of clusters of level 1.
Let ${\text{\bf C}}'_{0,1}=\{\{x_i\}\colon   x_i \in \Gamma \setminus
{\mathcal{C}} \text{ for all } {\mathcal{C}} \in {\text{\bf C}}_{1,1}\}$ and
${\text{\bf C}}_1={\text{\bf C}}_{1,1}\cup{\text{\bf C}}'_{0,1}$.
Note that $\text{\bf C}_{1,1}$ and  $\text{\bf C}'_{0,1}$ consist
of level 1 and level 0 clusters, respectively, and that the union
of all points in these clusters is exactly $\Gamma$. We label the
elements of ${\text{\bf C}}_1$ in increasing order as ${\mathcal{C}}_{1,j}, j \ge 1$.
For later use we also define $\text{\bf
C}_{0,1} = {\text{\bf C}}_0$. \noindent {\bf Notation.} In our
notation ${\mathcal{C}} ^1_j$ denotes the $j^{th}$ level 1-cluster,
and ${\mathcal{C}}_{1,j}$ denotes the $j^{th}$ element in $\text{\bf
C}_1$ (always in increasing order).

\medskip
\noindent {\bf Level k+1.} Let  $k \ge 1$ and assume that the
partitions ${\text{\bf C}}_{k'} = \{{\mathcal{C}}_{k',j}\colon \; j \ge 1
\}$, and the masses of the ${\mathcal{C}}_{k',j}$ have already been defined
for $k'\le k$, and satisfy the properties (\ref{2.z})-(\ref{2.two}) and that
${\text{\bf C}}_{k}$ consists of clusters ${\mathcal{C}}$ of
levels $\ell \in \{0,1,\dots,k\}$. I.e.,
\begin{equation}
\label{2.zz}
\text{\bf C}_{k}\subset \cup_{\l=0}^k  \text{\bf C}_{\l,\l},
\end{equation}
where for  $\l \ge 0$, ${\text{\bf C}}_{\l,\l}$ is the set of
level $\ell$ clusters. We assume, as before, that the labeling goes
in increasing order of appearance. Define
\begin{equation}
\label{2.4e}
{\text{\bf C}}_{k,k+1} = \{ {\mathcal{C}} \in {\text{\bf C}}_k \colon  m
({\mathcal{C}})\ge k+1 \}.
\end{equation}
Notice that ${\text{\bf C}}_{1,2}={\text{\bf C}}_{1,1}$ and
${\text{\bf C}}_{k,k+1}\subseteq\cup_{\ell=1}^k{\text{\bf
C}}_{\ell,\ell}$, if $k \ge 1$.

In the previous enumeration of ${\text{\bf C}}_{k}$, let $j_1 <
j_2 < \dots$ be the labels of the clusters in ${\text{\bf
C}}_{k,k+1}$, so that $\text{\bf C}_{k,k+1} = \{{\mathcal{C}}_{k,j_1},
{\mathcal{C}}_{k,j_2},\dots\}$. In ${\text{\bf C}}_{k, k+1}$ we consider
consecutive maximal $(k+1)$-runs, where we say that the clusters
${\mathcal{C}}_{k,j_s}, {\mathcal{C}}_{k,j_{s+1}}, \dots {\mathcal{C}}_{k,j_{s+n-1}}
\in {\text{\bf C}}_{k,k+1}$ form a {\it maximal $(k+1)$-run} of
length $n\ge 2$ if:
$$
d({\mathcal{C}}_{k,j_i}, \,  {\mathcal{C}}_{k,j_{i+1}})\, < \, M^{k+1}, \; i=s,
\dots ,s+ n-2,
$$
and in addition
$$
d({\mathcal{C}}_{k,j_i}, \,  {\mathcal{C}}_{k,j_{i+1}})\, \ge \, M^{k+1}  \; \;
\begin{cases} \text{for} \; \; i=s-1, i= s+n-1, \; &\text{if} \; j_s>1 \\
\text{for} \; \; i= s+n-1, \; &\text{if} \; j_s=1. \end{cases}
$$
Again it is immediate that $\mathbb{P}_p$-a.s. all $(k+1)$-runs are finite
and that infinitely many such runs exist. Again we can label them in
increasing order and write  $r^{k+1}_i = r^{k+1}_i ({\mathcal{C}}_{k,j_{s_i}},
{\mathcal{C}}_{k,j_{s_i+1}},$ $ \dots, {\mathcal{C}}_{k,j_{s_i+n_i-1}})$ for the $i$-th $(k+1)$-run, for suitable
$s_i,n_i$ such that $n_i\ge 2$ and $s_i+n_i \le  s_{i+1}$ for all
$i$. ($s_i,n_i$ have nothing to do with those in the previous
steps of the construction.) We set
$$
\alpha^{k+1}_i = \alpha({\mathcal{C}}_{k,j_{s_i}})\quad \text{and}\quad
\omega^{k+1}_i =\omega ({\mathcal{C}}_{k,j_{s_i+n_i-1}}),
$$
and call these the start-point and end-point of the run,
respectively. We define the span of the run
$$
\text{span}(r^{k+1}_i)=[\alpha^{k+1}_i, \, \omega^{k+1}_i],
$$
and  associate to it a cluster ${\mathcal{C}}^{k+1}_i$ of level
$k+1$, defined as
$$
  {\mathcal{C}}^{k+1}_i = \text{span}(r^{k+1}_i)\cap \Gamma.
$$


It is made up from the clusters ${\mathcal{C}}_{k,j_{s_i}},
{\mathcal{C}}_{k,j_{s_i+1}}, \dots, {\mathcal{C}}_{k,j_{s_i + n_i-1}}$. In this case,
the clusters ${\mathcal{C}}_{k,j_{s_i}}, {\mathcal{C}}_{k,j_{s_i+1}},
 \dots {\mathcal{C}}_{k,j_{s_{i}+n_i-1}}$ are called {\it constituents}
of ${\mathcal{C}}^{k+1}_i$. To the cluster ${\mathcal{C}}^{k+1}_i$ we
attribute the mass $m({\mathcal{C}}^{k+1}_i)$  by the following
rule:
\begin{equation}
\label{2.foura}
m ({\mathcal{C}}^{k+1}_i) = m ({\mathcal{C}}_{k,j_{s_i}})+
\sum_{s=s_i+1}^{s_i+n_i-1} (m ({\mathcal{C}}_{k,j_s}) - k) =
\sum_{s=s_i}^{s_i+n_i-1} m ({\mathcal{C}}_{k,j_s}) - k(n_i-1).
\end{equation}

The points $\alpha^{k+1}_i$ and $\omega^{k+1}_i $ will also be
called, respectively,
{\it start}- and {\it end-point} of the cluster ${\mathcal{C}}^{k+1}_i$, and
are also written as $\alpha({\mathcal{C}}^{k+1}_i)$ and
$\omega({\mathcal{C}}^{k+1}_i)$.

By ${\text{\bf C}}_{k+1,k+1}$ we denote the set of all level
$({k+1})$ clusters. Take $\text{\bf C}'_{k,k+1} = \{ {\mathcal{C}} \in
{\text{\bf C}}_{k}\colon {\mathcal{C}} \cap \text{span}(r^{k+1}_i) =
\emptyset, \; i=1,2, \dots \}$. Finally we define ${\text{\bf
C}}_{k+1} := {\text{\bf C}}_{k+1,k+1} \cup \text{\bf C}'_{k,k+1}$.
We label the elements of ${\text{\bf C}}_{k+1}$ as
${\mathcal{C}}_{k+1,j}, j \ge 1$, in increasing order. Note that a cluster in
${\text{\bf C}}_k$ is also a cluster in ${\text{\bf C}}_{k+1}$ if
and only if it is disjoint from the span of each maximal
$(k+1)$-run of length at least 2. Thus ${\text{\bf C}}_{k+1}$ may
contain some clusters of level no more than $k$, but
 some clusters (of level $\le k$) in ${\text{\bf C}}_{k}$
no longer appear in ${\text{\bf C}}_{k+1}$ (or any ${\text{\bf
C}}_{k+j}$ with $j \ge 1$).

Note also that in the formation of a cluster of level $(k+1)$,
clusters of mass at most $k$ might be incorporated while taking
the span of a $(k+1)$-run; they form what we call {\it dust}
(of level at most $k-1$) in between the
constituents, which have mass at least $k+1$.

This describes the construction of the $\bold C_k$. We next show by
induction that
\begin{equation}
\label{nested}
\bold C_k \text{ is a partition of $\Gamma$ and $\bold
C_k$ is a refinement of $\bold C_{k+1}$}
\end{equation}
for $k \ge 0$. This is clear
for $k=0$, since $\bold C_0$ is the partition of $\Gamma$ into
singletons. If we already know (\ref{nested}) for $0 \le k \le K$, then
it follows also for $k = K+1$ from the fact that clusters in $\bold
C_{k+1}$ are formed from the clusters in $\bold C_k$ by combining the
consecutive clusters between the start- and end-point of
a maximal $(k+1)$-run into one cluster. Thus it takes a
number of successive clusters in $\bold C_k$ and combines them into one
cluster. This establishes (\ref{nested}) for all $k$.

The definition of $\bold C_k$ shows that
$$
\bold C_{k,k} \subset \bold C_k \subset \bold C_{k,k} \cup \bold C'_{k-1,k}
\subset \bold C_{k,k} \cup \bold C_{k-1},
$$
from which we obtain by
induction that (\ref{2.zz}) holds, as well as
$$
\ell({\mathcal{C}}) = \min\{k\colon{\mathcal{C}} \in \text{\bf C}_k\}
$$
for any ${\mathcal{C}} \in \cup_{k\ge 1} \text{\bf C}_k$.




We use induction once more to show that for any $k \ge 0$
\begin{equation}
m({\mathcal{C}}) \ge \ell({\mathcal{C}}) +1  \text{ for any }
{\mathcal{C}} \in \bigcup_{0 \le\ell \le k} \mathbf{C}_\ell,
\label{2.fourc}
\end{equation}
and if ${\mathcal{C}}^{k+1}_i$ is formed from the constituents
${\mathcal{C}}_{k,j_{s_i}},\dots,{\mathcal{C}}_{k,j_{s_i+n_i-1}}$
with $n_i \ge 2$, then
\begin{equation}
m ({\mathcal{C}}^{k+1}_i) \ge \max_{s_i \le s \le s_i+n_i-1}
m({\mathcal{C}}_{k,j_s})+n_i-1 >  \max_{s_i \le s \le s_i+n_i-1}
m({\mathcal{C}}_{k,j_s}).
\label{2.fourb}
\end{equation}
Indeed, (\ref{2.fourc}) trivially holds for $k = 0$. Moreover, if
(\ref{2.fourc}) holds for $k \le K$, then (\ref{2.fourb}) for $k = K$
follows from the rule (\ref{2.foura}) (and $n_i \ge 2$). In turn,
(\ref{2.fourb}) and (\ref{2.fourc}) for $k \le K$ imply
\begin{equation}
m({\mathcal{C}}_i^{k+1}) \ge \max_{s_i \le s \le s_i+n_i-1}m({\mathcal{C}}_{k,j_s})+1,
\label{2.zy}
\end{equation}
and hence also (\ref{2.fourc}) for $k = K+1$.

So far we have shown that ${\text{\bf C}}_{k+1}$ is a
partition of $\Gamma$ which satisfies (\ref{2.z}) and (\ref{2.one}) with
$k$ replaced  by  $k+1$ (by the definition of ${\text{\bf
C}}'_{k,k+1}$ and induction on $k$).
We next show by an indirect proof that this is also true for
(\ref{2.two}). It is convenient to first prove the following claim:
\newline
 {\bf Claim.} If $t \ge 1$,
 ${\mathcal{C}} \in  \cup_{j \ge 0}\text{{\bf C}}_{t+j}$ and
$\ell({\mathcal{C}}) \le t$, then we have
\begin{equation}
{\mathcal{C}} \in \text{{\bf C}}_{s,s+1} \text{ for } \l({\mathcal{C}}) \le s \le
(m({\mathcal{C}})-1)\land t. \label{2.fourff}
\end{equation}
(see definition (\ref{2.4e})). To see this, define $\wh s$ as the
smallest $s \ge \l({\mathcal{C}})$ for which ${\mathcal{C}} \notin \text{{\bf
C}}_{s,s+1}$, and assume that $\wh s \le (m({\mathcal{C}})-1)\land t$. Then
$m({\mathcal{C}}) \ge \wh s+1$, so that we must have ${\mathcal{C}} \notin \text{{\bf C}}
_{\wh s}$. But also ${\mathcal{C}} \in\text{{\bf C}}_{\wh s - 1,\wh s}
\subseteq \text{{\bf C}}_{\wh s-1}$. (Note that $\wh s = \l({\mathcal{C}})$
cannot occur, because one always has ${\mathcal{C}} \in \mathbf{C}_{\l,\l+1}$ for
$\l = \l({\mathcal{C}})$, by
virtue of (\ref{2.fourc}).) But then it must
be the case that ${\mathcal{C}}$ intersects span $(r_i^{\wh s})$ for some
$i$. In fact, by our construction, ${\mathcal{C}}$ must then be a
constituent of some cluster in  $\text{{\bf C}}_{\wh s}$
corresponding to a maximal $\wh s$-run of length at least 2. But
then ${\mathcal{C}}$ does not appear in $\text{{\bf C}}_{\wh s +j}$ for any $j
\ge 0$, and in particular ${\mathcal{C}} \notin \cup_{j \ge 0} \text{{\bf C}}_{t+j}$,
contrary to our assumption. Thus, $\wh s \le (m({\mathcal{C}})-1) \land
t$ is impossible and our claim must hold.

We now turn to the proof of (\ref{2.two}). This is obvious for $k =
0$ or $k=1$. Assume then that (\ref{2.two}) has been proven for some
$k \ge 1$. Assume further, to derive a contradiction, that ${\mathcal{C}}'$
and ${\mathcal{C}}''$ are two distinct clusters in $\text{{\bf C}}_{k+1}$
such that $\min\{m({\mathcal{C}}'), m({\mathcal{C}}'')\} \ge  r$ but $d({\mathcal{C}}',{\mathcal{C}}'') <
M^r$  for some $r \le k+1$. Without loss of generality we take $r
= m({\mathcal{C}}') \land m({\mathcal{C}}'') \land (k+1)$. Let ${\mathcal{C}}'$ and ${\mathcal{C}}''$ have
level $\l'$ and $\l''$, respectively. Since these clusters belong
to ${\text{\bf C}}_{k+1}$ we must have $\max(\l',\l'') \le k+1$.
For the sake of argument, let $\l' \le \l''$. If $\l'=\l''=k+1$,
then $d({\mathcal{C}}', {\mathcal{C}}'') \ge M^{k+1}$, because, by construction, two
distinct clusters of level $k+1$ have distance at least $M^{k+1}$.
In this case we don't have $d({\mathcal{C}}',{\mathcal{C}}'') < M^r$, so that we may
assume $\l' < k+1$.

Now first assume that $r-1 \ge \max(\l',\l'')= \l''$. Since $r-1
\le k$ we then have by (\ref{2.fourff}) (with $t = k$) that ${\mathcal{C}}'$
and ${\mathcal{C}}''$ both belong to $\text{{\bf C}}_{r-1,r}$. If the
distance from ${\mathcal{C}}'$ to the nearest cluster in $\text{{\bf
C}}_{r-1,r}$ is less than $L^r$, then ${\mathcal{C}}'$ will be a constituent
of a cluster of level $r$ and ${\mathcal{C}}'$ will not be an element of
${\text{\bf C}}_{k+1}$. Thus it must be the case that the distance
from ${\mathcal{C}}'$ to the nearest cluster in $\text{{\bf C}}_{r-1,r}$ is
at least $M^{r}$. A fortiori, $d({\mathcal{C}}', {\mathcal{C}}'') \ge M^r$. This
contradicts our choice of ${\mathcal{C}}', {\mathcal{C}}''$.

The only case left to consider is when $r-1 < \max(\ell',\ell'') =
\ell''$. Since $r-1 = ( m({\mathcal{C}}') -1) \land (m({\mathcal{C}}'')-1) \land k \ge
\ell' \land \ell'' \land k = \ell'$ (by (\ref{2.fourc}); recall that $\ell' <
k+1$ now) this means $\l'
\le r-1 < \ell''$. We still have as in the last paragraph that ${\mathcal{C}}'
\in \text{{\bf C}}_{r-1,r}$, and that the distance between ${\mathcal{C}}'$
and the nearest cluster in $\text{{\bf C}}_{r-1,r}$ is at least
$M^r$.  By (\ref{2.one}) span $({\mathcal{C}}')$ and span $({\mathcal{C}}'')$ have
to be disjoint. For the sake of argument let us further assume
that ${\mathcal{C}}'$ lies to the left of ${\mathcal{C}}''$, that is, $\om({\mathcal{C}}') <
\alpha({\mathcal{C}}'')$. We claim that $\al({\mathcal{C}}'')= \al({\mathcal{C}})$ for some cluster
${\mathcal{C}} \in \text{{\bf C}}_{r-1,r}$. Indeed, the start-point of a
cluster of level $\wt \ell\ge 2$ equals the start-point of one of
its constituents, which belongs to $\text{{\bf C}}_{\wt \ell -1,\wt
\ell} \subseteq \text{{\bf C}}_{\wt \ell-1}$. Repetition of this
argument shows that $\al({\mathcal{C}}'')$ is also the start-point of a
cluster ${\mathcal{C}}$ which is a constituent of some cluster $\wh {\mathcal{C}}$ such that $s :=\ell({\mathcal{C}}) \le r-1$ but $t+1 := \ell(\wh {\mathcal{C}}) \ge r$. In particular, ${\mathcal{C}} \in \text{{\bf C}}_{t,t+1}$, so
that ${\mathcal{C}} \in \mathbf{C}_t$ and $m({\mathcal{C}}) \ge t+1 \ge r$.
Thus $\ell({\mathcal{C}}) \le r-1 \le (m({\mathcal{C}}) -1) \land t$.
It then follows from (\ref{2.fourff}) that ${\mathcal{C}} \in \text{{\bf
C}}_{r-1,r}$. As in the preceding case we then have

\begin{align}
d({\mathcal{C}}',{\mathcal{C}}'') & \ge \al({\mathcal{C}}'') - \om({\mathcal{C}}')= \al({\mathcal{C}}) - \om({\mathcal{C}}')\\
&\ge \text{ the distance from ${\mathcal{C}}'$ to the nearest cluster
in \bf C}_{r-1,r}\\
&\ge M^r.
\end{align}

Of course the inequality $d({\mathcal{C}}',{\mathcal{C}}'')$ remains valid if ${\mathcal{C}}'$
lies to the right of ${\mathcal{C}}''$, so that we have arrived at a
contradiction in all cases, and (\ref{2.two}) with $k$ replaced by
$k+1$ must hold. This completes the proof of (\ref{2.two}).

\medskip
\noindent {$\bf Construction  \; of$ ${\text{\bf C}}_{\infty}$.} Observe
that each $x\in \Gamma$ may belong to clusters of several levels, but
not to different clusters of the same level (see (\ref{2.one})). If
${\mathcal{C}}'$ and ${\mathcal{C}}''$ are two clusters of levels $\ell'$ and
$\ell''$, respectively, with $\ell' < \ell''$, then

\begin{equation}
\text{span } ({\mathcal{C}}') \cap \text{ span }({\mathcal{C}}'') \ne \emptyset \text{
implies
 span $({\mathcal{C}}') \subseteq$ span $({\mathcal{C}}'')$}.
\label{2.y}
\end{equation}
There will even have be a sequence ${\mathcal{C}}_0= {\mathcal{C}}', {\mathcal{C}}_1, \dots, {\mathcal{C}}_s,
{\mathcal{C}}_{s+1} = {\mathcal{C}}''$
such that ${\mathcal{C}}_i$ is a constituent of ${\mathcal{C}}_{i+1}, 0 \le i\le s$.
This follows from the fact that each $\mathbf{C}_k$
is a partition of $\Gamma$ and that $\mathbf{C}_k$ is a refinement of $\bold
C_{k+1}$. In fact, each element of $\mathbf{C}_{k+1}$ is obtained by
combining several constituents which are consecutive elements of
$\mathbf{C}_k$. (We allow here that an element of $\mathbf{C}_k$ is
already an element of $\mathbf{C}_{k+1}$ by itself.)
In turn, we see then from (\ref{2.zy}) that
$m({\mathcal{C}}'') > m({\mathcal{C}}')$. In particular, no point belongs to two different
clusters with the same mass. We shall use this fact in the proof of
the next lemma.

We define the random index
\begin{equation}
\kappa (x) = \sup\{\ell\colon x\in {\mathcal{C}} \;{\text{ for some}}\;
{\mathcal{C}} \in {\text{\bf C}}_{\ell,\ell}  \}.
\end{equation}
If we allow the value $\infty$ for $\kappa(x)$, then this index is
always well defined, since each $x \in \Gamma$ belongs at least to
the cluster $\{x\}$ of level 0.

\begin{lema} {\label{lemma1.1}} Assume that the sequence $\xi$ is
distributed according to $\mathbb{P}_p$. If
$p>0$ and $3 \le M < (64 p)^{-1/2}$ we have a.s. $\kappa(x) <
\infty$, for all $x \in \Gamma$.
\end{lema}

Lemma \ref{lemma1.1} can be used for the construction of ${\text{\bf
C}}_{\infty}$. It tells that with $\mathbb P_p$--probability one,
for each $x \in \Gamma$, there exists a cluster of
level $\kappa (x)\in \Bbb Z_+$ which contains $x$. This cluster is
unique, since the elements of ${\text{\bf C}}_{k,k}$ are pairwise
disjoint. We call it
 the {\it maximal cluster of} $x$ and denote it by ${\cal D}_x$.
Moreover, for $x,x' \in \Gamma$, if $x' \in {\cal D}_x$, then $\kappa
(x)= \kappa (x')$ and ${\cal D}_x = {\cal D}_{x'}$. Indeed, $\kappa(x)
\ne \kappa(x')$ would contradict (\ref{2.y}) and the definition of
$\kappa$, while $\kappa (x)= \kappa (x')$ but ${\cal D}_x \ne {\cal D}_{x'}$ is impossible by (\ref{2.one}).

Take $\hat x_1 = x_1 \in \Gamma = \{ x_j\}_{j\ge 1} $ and define
${\mathcal{C}}_{\infty ,1} = {\cal D}_{\hat x_1}$. Having defined ${\mathcal{C}}_{\infty ,j} = {\cal D}_{\hat x_j}$ for $j=1,\dots,k$, we set $\hat
x_{k+1} = \min \{x_j \in \Gamma \colon x_j \notin \cup_{i=1}^k {\mathcal{C}}_{\infty ,i} \}$, and
${\mathcal{C}}_{\infty ,k+1} = {\cal D}_{\hat
x_{k+1}}$. Define ${\text{\bf C}}_{\infty} = \{ {\mathcal{C}}_{\infty ,k}
\}_{k \ge 1}$. Clearly, $\Gamma = \cup_{k \ge 1}\; {\mathcal{C}}_{\infty
,k}$. It is also routine to check that ${\text{\bf C}}_{\infty}$
satisfies (\ref{2.onea}) and (\ref{2.a}). As for (\ref{2.four}), this
follows from (\ref{2.two}) and the fact that ${\cal D}_x \in \text{\bf
C}'_{k,k+1} \subseteq \text{\bf C}_{k+1}$ for all $k \ge \kappa(x)$
(by the definitions of $\kappa(x)$ and $\text{\bf C}'_{k,k+1}$).
\medskip

The proof of Lemma \ref{lemma1.1} involves an exponentially small
upper bound (in $k$) for the probability of having a cluster
starting at a fixed point and having mass $k$.

Let $\Xi (p)$ denote the event of full probability in
Lemma \ref{lemma1.1} where the above construction of
${\text{\bf C}}_{\infty}$ is well set. We then define
\begin{equation}
\label{chi}
\chi(\xi)=\inf\{k\ge 0\colon d(\mathcal{C},0) \ge M^{m(\mathcal{C})} \text{ for all } \mathcal{C} \in \text{\bf C}_{\infty} \text { with } m(\mathcal{C})>k\}
\end{equation}
with $\chi(\xi)=\infty$ if the above set is empty or $\xi \notin
\Xi(p)$.
\medskip

\noindent As a consequence of Lemma \ref{lemma1.1} one has
\begin{proposition}\label{KSV} Let $p <\frac{1}{64M^2}$.
Then
\medskip

\centerline{$\mathbb{P}_p(\xi\colon \chi(\xi)<\infty)=1$ and $\mathbb{P}_p (\chi(\xi) =0)>0 $.}
\end{proposition}

Having in mind an application of the above grouping procedure to the
proof of the main result, and using the notation introduced
above, to each binary sequence $\xi\in\Xi(p)$ we associate the
sequence $\psi^\xi\in\Psi$ defined as follows:
\begin{equation}
\label{psinova}
\psi^\xi_{i}=
\begin{cases} m({\mathcal{C}}_{\infty ,j})
,&\mbox{ if } i=\hat x_j-\sum_{t=0}^{j-1}\text{diam}({\mathcal{C}}_{\infty ,t})\,;\\ 0, &\mbox{otherwise },
\end{cases}
\end{equation}
where $\text{diam}(C)$ denotes the diameter of $C$ (for the Euclidean distance).
As an immediate corollary of Proposition \ref{KSV} we have:

\begin{corolario}\label{sec}
If $0< p <\frac{1}{64M^2}$, then $\mathbb{P}_p\{\xi\in\Xi(p)\colon
\psi^\xi\in \Psi_M\}>0$.
\end{corolario}

Now we give an upper bound on the cardinality of a cluster in terms of its mass.
For this we simply use the estimate on the
cluster diameter obtained in the proof of Lemma \ref{lemma1.1}.

\begin{lema}\label{cotadeuns}
Let $0< p <\frac{1}{64M^2}$ and $\xi \in \Xi(p)$. Then
\begin{equation*}
\sup\{|{\mathcal{C}}|\colon {\mathcal{C}}\in {\mathbf C}_\infty, m({\mathcal{C})}=k \} \le 3M^{k-1}
\end{equation*}
for all $k\in\NN$.
\end{lema}

\proof See \eqref{2.6} and \eqref{2.25} in the appendix.\endproof

\section{Proof of Theorem \ref{principal}
\label{IV}}

\noindent The following Definition and Lemma will play important role
in the proof of Theorem \ref{principal}:
\begin{definicao}
For $\zeta, \psi \in \Psi$ we say that $\psi \preceq_{_{M}}\zeta$
if for any $j \ge 0$ the following holds:
\begin{align}
\psi_j = 0 & \Leftrightarrow \zeta_j =0, \notag
\\
\psi_j = k & \Rightarrow k \leq \zeta_j \leq 3M^{k-1}.
\end{align}
\end{definicao}

\begin{lema}
\label{emagrecimento} Let $p< \frac{1}{64M^2}$ and $\xi \in \Xi_\infty$
such that $\chi(\xi)=0$. Let $j_1(\xi)< 1+j_2(\xi) < 2+j_3(\xi)<\dots $
denote the ordered
elements of $\cup_{i\ge
1}(\text{span}(\mathcal{C}_{\infty,i})\setminus
\mathcal{C}_{\infty,i})$. Set $\xi^{(0)}=\xi$ and for each $n \ge 1$
\begin{equation}
\xi^{(n)} := \triangle^{\mathbf{0}}_{j_n(\xi)}(\xi^{(n-1)}).
\end{equation}
The limit $\tilde\xi:= \lim_{n \to \infty} \xi^{(n)}$ exists
in $\Xi_\infty$ and
\begin{equation*}
\psi^\xi\preceq_{_{M}}f(\tilde \xi)=:\tilde \psi,
\end{equation*}
where $\psi^\xi$  is given by \eqref{psinova} and the function $f$ was defined
in \eqref{f1}--\eqref{f2}.
\end{lema}

\begin{proof}
The proof follows at once from the previous construction and Lemma
\ref{cotadeuns}.
\end{proof}



\medskip

\noindent We now construct the binary sequence $\eta$ that appears in
Theorem \ref{principal}. It is
obtained from $\zeta{(L)}$, see \eqref{hierarquica},
replacing each entry $(\zeta{(L)})_j=k$ by a string of $3M^{k-1}$
consecutive ones, with $M=3(L+1)$, and correspondingly shifting the
rest of the sequence to the right. That is, fix $L \geq 2$,
 and
define $\tilde\zeta{(L)} \in \Psi$  as follows:
\begin{equation}
\label{tildezeta} (\tilde\zeta{(L)})_j =
\begin{cases}
3M^{k-1} ,&\mbox{ if } L^k |j \; \text{ and }  \; L^{k+1}\nmid j,
\\
0\ , &\mbox{ if } L \nmid j,
\end{cases}
\end{equation}
for all $j \ge 1$, with $M=3(L+1)$. We then let
$\eta(L)$ be the unique element of $\Xi_\infty$ such that $\tilde \zeta(L)=f(\eta(L))$,
where
 $f$ was defined by \eqref{f1}--\eqref{f2} in Section \ref{II}.

\medskip

\noindent We can now state and prove the following result.

\begin{teorema}\label{final} Let $L\geq 2$. If $p<\frac{1}{576(L+1)^2}$ and the deterministic binary
sequence $\eta{(L)}$ is defined as above, then
\begin{equation}
\mathbb{P}_{p}\{\xi\in\Xi\colon (\eta{(L)},\xi)\mbox{ is
compatible}\}>0.
\end{equation}
\end{teorema}

\begin{proof} Let $L \ge 2$ and $M=3(L+1)$. If $p<\frac{1}{64M^2}$,
Corollary \ref{sec} says that $\mathbb{P}_p\{\xi\in\Xi(p)\colon
\psi^\xi\in \Psi_{M}\}>0$, and from Corollary \ref{pesos} it
follows that
\begin{equation*}
\mathbb{P}_p\{\xi\in\Xi(p) ;(\zeta{(L)},\psi^\xi)\mbox{ is
compatible}\}>0.
\end{equation*}
On the other hand, using Lemma \ref{emagrecimento} it is simple to check
that if the pair $(\zeta{(L)},\psi^\xi)$ is compatible, then so is
$(\tilde \zeta{(L)},\tilde \psi)$. The proof then follows
by recalling Proposition \ref{pesos2}.
\end{proof}

\medskip



To conclude the proof of
Theorem \ref{principal}, it now remains to check that the set of zeroes $\mathcal{Z}(L)$ in the sequence
$\eta(L)$ in Theorem \eqref{final} has discrete Hausdorff dimension that tends to
one as $L$ tends to infinity. As one verifies at once, its
asymptotic density is zero.

We use the notion of discrete Hausdorff dimension, as introduced in \cite{BT}, and related results
from \cite{ACHR}. Among the simplest measures of the asymptotic size of a set $A\subseteq \ZZ$, let
us recall the following:

\begin{definicao} \label{massdim} The {\it lower} and {\it upper} mass dimensions of
$A\subseteq \ZZ$ can be defined as
\begin{equation}
{\rm{dim}}_{\rm{LM}}(A) = \liminf_{n\to \infty} \frac{\log |A\cap[-n/2,
n/2)|}{\log n},
\end{equation}
\begin{equation}
{\rm{dim}}_{\rm{UM}}(A) = \limsup_{n\to \infty} \frac{\log |A\cap[-n/2,
n/2)|}{\log n},
\end{equation}
and if ${\rm{dim}}_{\rm{LM}}(A) = {\rm{dim}}_{\rm{UM}}(A)$, we call it the mass
dimension of $A$.
\end{definicao}

\noindent A simple computation shows that\footnote{To match exactly to the previous definition, we may take the two-sided version of the hierarchical sequence.}
$${\rm{dim}}_{\rm{LM}} (\mathcal{Z}{(L)})={\rm{dim}}_{\rm{UM}} (\mathcal{Z}{(L)}) = \frac{\log L}{\log (3(L+1))}.$$

\noindent To recall the definition of discrete Hausdorff dimension  introduced in \cite{BT}
we need some notation:


\noindent Let $\mathcal{I}$ denote the set of all intervals $[x,y),
\; x,y \in \ZZ$. Given positive integers $r, \, n, \; r\ge 2$, we
denote
\begin{align}
I_1^{(r)} & = [-r, r) \notag
\\
I_n^{(r)} & = [-r^n, r^n) \setminus [-r^{n-1}, r^{n-1}), \; n \geq
2. \notag
\end{align}
Given $\alpha >0$ and $A,F \subseteq \ZZ$, set
\begin{equation}
\nu_\alpha (A,F) = \inf \Big \{ \sum_i (d(B_i))^\alpha : B_i \in
\mathcal{I}, A \cap F \subseteq \cup_i B_i  \Big \} (d(F))^{-\alpha},
\end{equation}
and
\begin{equation}
m_\alpha^{(r)} (A) = \sum_{n=1}^\infty \nu_\alpha (A,I_n^{(r)}).
\end{equation}
\begin{definicao} \label{hausdorff} The discrete Hausdorff dimension of a set $A$ is defined by
\begin{equation}
{\rm{dim}}_{\rm{H}} (A) = \inf \{ \alpha >0 \colon m_\alpha^{(r)}(A) < \infty \}.
\end{equation}
\end{definicao}
\smallskip

\noindent Another useful notion is

\begin{definicao} \label{entropy} The upper entropy index of a set $A$ is defined by
\begin{equation}
\Delta(A) = \inf \{ \alpha > 0: \max_{1 \le d \le
r^{n(1-\epsilon)}} \{(dr^{-n})^\alpha N(d, A \cap I_n^{(r)}) \to 0
\text{ for each } \epsilon >0     \} \},
\end{equation}
where $N(d, A)$ denotes the maximum number of disjoint
intervals in $\mathcal{I}$ of length $2d$ and with centers in $A$.
\end{definicao}

\noindent {\bf Remark.} It is easy to see that
\begin{equation}
0 \le {\rm{dim}}_{\rm {H}} (A) \leq {\rm{dim}}_{\rm{UM}} (A) \le \Delta(A) \le 1.
\end{equation}

There is also a related notion introduced in \cite{BT}, called {\it discrete
packing dimension}, denoted by ${\rm{dim}}_{\rm{p}}$, which uses the concept of
packing measure suitably adapted to the discrete setup (see \cite{BT}, Sect. 3, p. 130).
This notion was used by Barlow and Taylor \cite{BT} to define fractal sets in $\ZZ$, as those for
which there is equality of the Hausdorff and packing dimensions. On the other hand,
the upper entropy index $\Delta$ always coincides with the packing dimension,
as proven in \cite{BT} (Lemma 3.1, p. 131). Thus, one may equivalently say:
\begin{definicao}
\label{fractal} \cite{BT}  The set $A$ is called fractal if
${\rm{dim}}_{\rm{H}} (A) = \Delta(A) = {\rm{dim}}_{\rm{p}} (A)$.
\end{definicao}

\smallskip

\noindent To determine the Hausdorff dimension of $\mathcal{Z}(L)$ we use
the following result from \cite{ACHR} (a corollary of Theorem 4.1 in \cite{BT})

\begin{proposition} \label{C2} {\rm (Part of Corollary 2 of \cite{ACHR})}
Let $A \subseteq \mathbb{Z}$. Suppose there are positive constants
$c, \;c'$ such that for each large $n \in \mathbb{N}$ and all
integers $x$,
\begin{equation}
\label{eqC2}
| A \cap[-r^n, r^n) | \ge c'r^{n\alpha} \text{ and } | A \cap[x-r^n,
x+r^n) | \le cr^{n\alpha}.
\end{equation}
Then ${\rm{dim}}_{\rm{H}} (A) = {\rm{dim}}_{\rm{UM}} (A)$.
\end{proposition}


\noindent To see $\mathcal{Z}(L)$ is a
fractal we use the following statement from \cite{ACHR}:

\begin{proposition} \label{T3} (Theorem 3 of \cite{ACHR})
Let $A \subseteq \mathbb{Z}$. Suppose there are positive constants
$c, \;c'$ such that for all $x \in A$ and all $n \in \mathbb{N}$ ,
\begin{equation}
\label{propfractal}
| A \cap[x-r^n, x+r^n) | \ge c'r^{n\alpha} \text{ and } | A
\cap[-r^n, r^n) | \le cr^{n\alpha}.
\end{equation}
Then $\Delta(A) = {\rm{dim}}_{\rm{UM}} (A)$.
\end{proposition}

\noindent{\it Conclusion of the proof of Theorem \ref{principal}.}
Straightforward computations show that both conditions \eqref{eqC2} and
\eqref{propfractal} are satisfied for $\mathcal{Z}(L)$. In particular,
it is a fractal of Hausdorff dimension $\frac{\log L}{\log (3(L+1))}$. This
concludes the proof. \qed

\section{Appendix. Proof of Lemma \ref{lemma1.1}}

\noindent {\it Proof of Lemma \ref{lemma1.1}.} $\kappa (x)=
+\infty$ can occur only if there exists an infinite increasing
subsequence of indices $\{k_i \}_{i\ge 1}$ such that the point $x$
becomes ``incorporated'' into some cluster of level $k_i$ for all
$i\ge 1$. We will show that
\begin{equation}
\label{2.3aa}
\mathbb{P}_p\left( x \text{ belongs to an infinite sequence of
clusters}\right)=0.
\end{equation}
Notice that each cluster of level $k$ necessarily has mass at
least $k+1$ and no point belongs to two different  clusters of the
same mass, as already observed. Setting
\begin{equation}
\label{2.4}
A_k(x)=[x \text { belongs to a cluster of mass } k],
\end{equation}
we shall show that for each fixed $x$
\begin{equation}
\label{2.4a}
\mathbb{P}_p(A_k(x) \text { i.o. in } k)=0,
\end{equation}
which will prove \eqref{2.3aa}.

We will carry out the proof  in two steps. All constants $c_i$
below are strictly positive and independent of $k$. First we
estimate the probability that a given point $z\in \Bbb Z_+$ is the
start-point of a cluster of mass $k\ge 2$. Specifically, we show
that

\begin{equation}
\label{2.5}
\mathbb{P}_p \left(\exists\, {\mathcal{C}} \in \cup_{\ell \ge 1}{\mathbf C}_\ell
\colon \alpha({\mathcal{C}})=z, m({\mathcal{C}})=k\right) \le c_1 e^{-c_2 k}
\end{equation}
for some strictly positive constants $c_1, c_2 >0$ and for each
fixed $k$. In fact we can take $c_2 > \log L$ so that
\begin{equation}
\label{2.7}
c_1(L^k+1)e^{-c_2k} \le 2c_1e^{-c_3k}
\end{equation}
for some constant $c_3 > 0$. This is the most involved part of the
proof.
In the second step of the proof we show  that if $\mathcal{C} \in
\cup_{\ell \ge 1} {\mathbf C}_\ell$ and $m({\mathcal{C}}) = k$, then
\begin{equation}
\label{2.6}
\text{diam} \big( {\mathcal{C}} \big) \; < 3L^{k-1}.
\end{equation}
Due to \eqref{2.6} we will have the following inclusion:
\begin{equation}
\label{2.8}
A_k(x) \subseteq \big[\exists z\in [x- L^k, x]\colon \alpha({\mathcal{C}})=z, \; {\text {for some}}\; {\mathcal{C}} \in \cup_{\ell \ge 1} {\mathbf
C}_\ell \; {\text {and }}  m({\mathcal{C}})=k \big].
\end{equation}
>From \eqref{2.5}, \eqref{2.7} and \eqref{2.8} we will have
\begin{equation}
\label{2.9}
\mathbb{P}_p \big(A_k(x)\big) \le  c_1(L^k +1) e^{-c_2 k} \le 2c_1e^{-c_3k},
\end{equation}
which, by the Borel-Cantelli lemma, gives  \eqref{2.4a}, and so
\eqref{2.3aa}.

Let us now prove \eqref{2.5}, where $k \ge 2$ and $z \in \Bbb Z_+$.
To any given cluster ${\mathcal C} \in \cup_{\ell\ge 1}{\mathbf C}_\ell$ we
associate a ``genealogical weighted tree". It describes the
successive merging processes which lead to the creation of $\mathcal{C}$, i.e., it tells the levels at which some clusters form runs,
merging into larger clusters and how many constituents entered
each run, down to level 1, and finally the masses of such level 1
clusters. So we represent it as a tree with the root corresponding
to ${\mathcal C}$; the leaves correspond to clusters of level 1, which
are the basic constituents at level 1. This weighted tree gives
the basic information on the cluster, neglecting what was
incorporated as ``dust", on the way.

More formally, we construct the tree iteratively. The root of the
tree corresponds to the cluster $\mathcal C$. If this cluster is of
level 1, the procedure is stopped. For notational consistency such
a tree will be called a 1-leaf tree. To the root we attribute the
index 1, as well as another index which equals the mass of the
cluster.

If the resulting cluster $\mathcal C$ is of level $\ell\,>1$,
 we attribute to the root the
index $\ell$ and add to the graph $n_1$ edges (children) going out
from the root, where $n_1 \ge 2$ is the number of constituents
which form the $\ell$-run leading to $\mathcal{C}$. Each
 endvertex of a newly added edge will correspond to a constituent of
the run, i.e., if $\mathcal C$ has constituents
 ${\mathcal C}_{\ell-1,i_1},\dots, {\mathcal C}_{\ell-1,i_{n_1}} \in {\mathbf
C}_{\ell-1,\ell}$, for suitable $i_1,\dots, i_{n_1}$, then there
is a vertex at the end of an edge going out from the root
corresponding to $\mathcal{C}_{\l-1, i_j}$ for each $j
 = 1, \dots, n_1$. If the constituent corresponding to a given
endvertex is a level 1-cluster, the procedure at this endvertex is
stopped (producing a leaf on the tree), and to this leaf we
attribute an index, which equals the mass of the corresponding
constituent.

If a given endvertex corresponds to a cluster $\wt {\mathcal C}$ of
level $\ell'$ with $1< \ell' < \ell$, then to this endvertex we
attribute the index $\ell' $, and add to the graph $n_2$ new edges
going out of this endvertex, where $n_2$ is the number of
constituents of $\wt {\mathcal C}$ in ${\text {\bf C}}_{\ell'-1,\ell'}$
which make up $\wt {\mathcal{C}}$.

The procedure continues until we reach the state that all
constituents corresponding to newly added edges are level 1
clusters. In this way we obtain a tree with the following
properties:

 i) each vertex of the tree has either 0 or at least two offspring;
in case of 0 offspring we say that the vertex is a {\it leaf} of
the tree. Otherwise we call it a {\it branch node}.

 ii) to each branch node $x$ we attribute an index $\ell_x$; these
indices are strictly decreasing to 1 along any selfavoiding path
from the root to a leaf of the tree.

iii) to each leaf is associated a mass $m \ge 1$. This defines a
map

\begin{equation*}
\gamma\colon \mathcal C \in \cup_{\ell\ge 1}{\mathbf C}_\ell \mapsto
\gamma(\mathcal C) \equiv (\Upsilon (\mathcal C) , \bar l(\mathcal C) , \overline m
(\mathcal C)),
\end{equation*}
where $\Upsilon (\mathcal C)$ is a finite tree with $\mathcal L (\Upsilon
(\mathcal C))$ leaves and $\mathcal N (\Upsilon (\mathcal C))$ branching
nodes. We use the following notation:
\medskip \noindent
$\overline l(\mathcal C)= \{ \ell_1(\mathcal C), \dots , \ell_{\mathcal N
(\Upsilon (\mathcal C))}(\mathcal C) \}$  is a multi-index with one
component for each branching node of $\Upsilon (\mathcal C)$, which
indicates the level at which  branches ``merge" into the cluster
corresponding to the node;

\medskip
\noindent $\overline m (\mathcal C)= \{ m_1(\mathcal C), \dots , m_{\mathcal L
(\Upsilon (\mathcal C))}(\mathcal C) \}$ a multi-index with one component
for each leaf of $\Upsilon (\mathcal C)$, which gives to the mass of
the cluster corresponding to the leaf;

\medskip
\noindent $\bar n (\mathcal C)= \{ n_1(\mathcal C), \dots , n_{\mathcal N
(\Upsilon (\mathcal C))}(\mathcal C) \}$  is a multi-index with one
components for each vertex of $\Upsilon (\mathcal C)$, which gives the
degree of the vertex minus 1. Note that $\bar n (\mathcal C)$ is determined by
$\Up(\mathcal C)$.

To lighten the notation, we will omit the argument $\mathcal C$ in
situations where confusion is unlikely. Thus we occasionally write
$\gamma(\mathcal C) \equiv (\Upsilon, \bar l, \overline m)$ instead of
$(\Upsilon (\mathcal C) , \bar l(\mathcal C) , \overline m (\mathcal C))$.

In order to prove \eqref{2.5} we decompose the event
\begin{equation}
\big [\exists\, \mathcal C \in \cup_{\ell\ge 1}{\mathbf C}_\ell
\colon \alpha(\mathcal C) = z, m(\mathcal C)=k\big] \label{2.12}
\end{equation}
according to the possible values for $\gamma(\mathcal C)$; we shall
abbreviate the number of leaves of $\Up(\mathcal C)$ by $\mathcal L$. Since the
resulting cluster $\mathcal C$, obtained after all merging process ``along
the tree'', has mass $k$, it imposes the following relation
between the multi-indices $\overline m$ and $\bar l$:
\begin{equation}
\sum_{i=1}^{\mathcal L} m_i - \sum_{j=1}^{\mathcal N} (n_j-1)(\ell_j-1) = k
\label{2.12aa}
\end{equation}
Here the first sum runs over all leaves, while the second sum runs
over all branching nodes.
This relation follows from \eqref{2.foura} by induction on the number of
vertices, by writing the tree as the ``union'' of the root and
the subtrees which remain after removing the root. We note that
$\Up$ also has to satisfy
\begin{equation}
\sum_{j=1}^\mathcal N  (n_j-1) = \mathcal L-1, \label{2.14aa}
\end{equation}
because it is a tree, as one easily sees by induction on the
number of leaves. This implies the further restriction
\begin{equation*}
\sum_{i=1}^\mathcal L m_i \ge k+\mathcal L-1,
\end{equation*}
because $\l_j \ge 2$ in each term of the second sum in
\eqref{2.12aa} (recall that we stop our tree construction at each
node corresponding to a cluster of level 1). Thus the probability
of the event in \eqref{2.12} equals to

\begin{equation}
\sum_{r\ge 1} \sum_{\Upsilon\!\colon\!\mathcal L (\Upsilon) =r }
 {\sum_{\bar l, \overline m}}^{\Upsilon}  \mathbb{P}_p \big(\exists\, \mathcal C \in
\cup_{\ell\ge 1}{\mathbf C}_\ell \colon z=\al(\mathcal C), m(\mathcal
C)=k, \gamma (\mathcal C)
 = (\Upsilon, \bar l, \overline m)\big ),
\label{2.12bb}
\end{equation}
where the third sum $\sum_{\bar l, \overline m}^{\Upsilon}$ is
taken over all possible values of $\bar l, \overline m$,
satisfying \eqref{2.12aa}.

A decomposition according to the value of the sum $\sum_i m_i$,
shows that the expression \eqref{2.12bb} equals
\begin{equation}
\sum_{_{r\ge 1}} \sum_{_{\substack{\Upsilon\!\colon\! \\ \mathcal L (\Upsilon) =r}}} \sum_{_{s \ge r-1}} \sum_{_{\substack{\overline m\!\colon\!\! \\ \sum_i m_i  \\
= k +s }}} \sum_{ \bar l}  \mathbb{P}_p \left( \exists \;  \mathcal C \in
\cup_{_{\ell\ge 1}}{\mathbf C}_\ell \colon \alpha(\mathcal C)=z, \; m({\mathcal
C}) = k , \gamma (\mathcal C) = (\Upsilon, \bar l, \overline m)\right),
\label{2.12cc}
\end{equation}
the sum $\sum_{\bar l}$ being taken over possible choices
of $\bar l$ such that $\sum_j (n_j -1)(\ell_j-1)  = s$. The multiple
sum in \eqref{2.12cc} can be bounded from above by
\begin{equation}
 \sum_{r\ge 1} \sum_{\substack{\Upsilon\!\colon\! \\ \mathcal L (\Upsilon) =r}}
\sum_{s \ge r-1} \sum_{\substack{\overline m\!\colon\! \\ \sum_i m_i  = k +s}}
\sum_{ \bar l} p^{k+s} L^{k+2s}. \label{2.14}
\end{equation}
Indeed, for fixed $z, k$ and $(\Up, \bar l, \overline m)$, the
probability
\begin{equation*}
\mathbb{P}_p\left(\exists\;  \mathcal C \colon \alpha(\mathcal C)=z, \;
m ({\mathcal C}) = k , \gamma (\mathcal C) = (\Upsilon, \bar l, \overline m)
\right)
\end{equation*}
is easily estimated by the following argument: the
probability to find a level 1 cluster  of mass $m_i$ which
corresponds to some leaf of the tree, and which starts at a given
point $x$, is bounded from above by $p^{m_i}L^{m_i -1}$. Indeed,
such a cluster has to come from a maximal level 1 run $x_s,
x_{s+1},\dots, x_{s+m_i-1}$ of elements of $\Gamma$, with $x_s = x$
and $x_{j+1}-x_j \le L$ for $j=s,\dots,s+m_i-2$. The number of
choices for such a run is at most $L^{m_i-1}$, and given the
$x_j$, the probability that they all lie in $\Gamma$ is $p^{m_i}$.
Similarly, the probability to find two level 1 clusters of mass
$m_{i_1}$ and $m_{i_2}$ which merge at level $\ell_j$ can be bounded
above by $p^{m_{i_1}}L^{m_{i_1} -1}p^{m_{i_2}}L^{m_{i_2}
-1}L^{\ell_j}$. The factor $L^{\ell_j}$ here is an upper bound for the
number of choices for the distance between the two clusters; if
they are to merge at level $\ell_j$, their distance can be at most
$L^{\ell_j}$. Iterating this argument we get that
\begin{eqnarray*}
\mathbb{P}_p \left( \exists \mathcal C \in \cup_{\ell\ge 1}{\mathbf C}_\ell
\colon \alpha(\mathcal C)=z, \; m({\mathcal C}) = k , \gamma (\mathcal C) =
(\Upsilon, \bar l, \overline m)\right) \\\le p^{\sum_i m_i}
L^{\sum_i (m_i-1)} L^{\sum_j (n_j-1)\l_j},
\end{eqnarray*}
and taking into account that
\begin{eqnarray*}
\sum_i (m_i-1)  + \sum_j (n_j-1)\ell_j &=& \! \! \! \sum_i m_i + \sum_j
(n_j-1)(\ell_j-1) -\sum_i 1 + \sum_j (n_j-1) \\  &=& k+s + s - r +
\sum_j (n_j-1),
\end{eqnarray*}
as well as \eqref{2.14aa}, we get the bound \eqref{2.14}.

The number of terms in the sums of \eqref{2.14} over $\overline m$
and $\bar l$ are respectively bounded by $2^{k +s}$ and $2^{s}$
(since $\sum_j(\ell_j-1) \le \sum_j(n_j-1)(\ell_j-1)= s$ and $\ell_j \ge
2$). Thus we can bound \eqref{2.14} from above by
\begin{eqnarray}
\label{2.16}
 &\sum_{r\ge 1} \sum_{\Upsilon\!\colon\! \mathcal L (\Upsilon) =r}
 \sum_{s \ge r-1}
2^{k +s} 2^{s}  p^{k+s} L^{k+2s} \nonumber\\ \nonumber
&\le   (2p L)^k \sum_{r\ge 1} \sum_{\Upsilon \!\colon \mathcal L (\Upsilon) =r} \sum_{s \ge r-1}
  (4 p L^2)^{s}\\
&\le   (2p L)^k \sum_{r\ge 1} \sum_{\Upsilon \!\colon\! \mathcal L
(\Upsilon) =r}
  \frac{(4 p L^2)^{r-1}}{1-4 p L^2},
\end{eqnarray}
provided we take $4 p L^2 <1 $. Now the number of planted plane
trees of $u$ vertices is at most $4^u$ (see \cite{HPT}). Our trees have $r$ leaves, but all vertices which
are not leaves have degree at least 3 (except, possibly, the
root). Thus, by virtue of \eqref{2.14aa}, these trees have at most
$2r$ vertices. The number of possibilities for $\Up$ in the last
sum is therefore at most $\sum_{u=r+1}^{2r} 4^u \le \frac43 4^{2r}
\le 2\cdot 4^{2r}$. It follows that \eqref{2.16} is further bounded
by
\begin{equation*}
2\frac{(2p L)^k}{1-4 p L^2} \sum_{r\ge 1} 4^{2r}
  (4 p L^2)^{r-1}
=   \frac{32(2p L)^k}{1-4 p L^2} \sum_{r\ge 1}
  (64 p L^2)^{r-1}.
\end{equation*}
If we take $64p L^2 <1 $, this can be bounded by
$$
\frac{32(2p L)^k}{(1-4 p L^2)(1-64 p L^2)},
$$
which proves \eqref{2.5} and \eqref{2.7} with $c_2 = -\log(2p) -
\log  L > \log L$ for our choice of $p,L$.
It remains to show \eqref{2.6}. It is trivially correct for $k=1$;
in fact a cluster of mass 1 has to be a singleton by
\eqref{2.fourb}. We will use induction on $k$. Assume \eqref{2.6}
holds for all clusters with mass at most $k-1$, where $k\ge 2$.
Let $\mathcal C$ be a cluster with $m(\mathcal C)=k$ and of level $\ell$.
Thus $\mathcal C \in {\mathbf C}_{\ell,\ell}$, and $1 \le \ell \le
k-1$, by virtue of \eqref{2.fourc}. If $\ell=1$ then diam$(\mathcal C)
\leq (k-1)L < 3L^{k-1}$, for $k\ge 2$, provided we take $L \ge 2$.\
If $\ell \ge 2$, then there exist $n\ge 2$, and $\mathcal C_{\ell-1,
i_1}, \dots \mathcal C_{\ell-1, i_n} \in {\mathbf C}_{\ell-1
,\ell}$ such that $\mathcal C$ is made up from the constituents $\mathcal
C_{\ell-1 , i_1}, \dots \mathcal C_{\ell-1 , i_n}$ (where, for
simplicity, we have omitted the indication of the level of the
constituents). If $m_j = m (\mathcal C_{\ell -1, i_j})$, then $m_j \ge
\ell$ (by \eqref{2.fourc}), and from \eqref{2.foura} we see that $m_j
\le k-n+1$ for each $j$. From this and the induction hypothesis we
get
\begin{equation}
\text{diam}(\mathcal C) \le \sum_{j=1}^n {\text{diam}}(\mathcal C_{\ell-1,
i_j}) + (n-1)L^{\ell} < 3nL^{k-n}+(n-1)L^{k-n+1}\le 3L^{k-1}
\label{2.25}
\end{equation}
for all $L\ge 3$ and $n \ge 2$. This proves \eqref{2.6} and the
lemma. \qed

\bigskip

\noindent{\it Proof of Proposition \ref{KSV}.} From \eqref{2.5} and \eqref{2.7} we have

\begin{eqnarray}
\mathbb{P}_p(\exists \; {\mathcal C}_{\infty,j}\colon \; m({\mathcal C}_{\infty,j}) > k,
d({\mathcal C}_{\infty,j}, 0) < L^{m({\mathcal C}_{\infty,j})})\le \sum_{m >  k} c_1 L^me^{-c_2m}
\end{eqnarray}
which tends to zero as $k \to \infty$, proving that $\chi(\cdot) <\infty$ a.s.

Now if $\xi\in \Xi(p)$ is such that $\chi(\xi)$ is finite and non-zero, then there exists a unique cluster $\mathcal
C^\ast \in {\mathbf C}_{\infty}(\xi)$ such that $m (\mathcal{C}^\ast)
= \chi(\xi)$ and $d(\mathcal{C}^\ast, 0) < L^{\chi(\xi)}$. The
existence of ${\mathcal{C}}^*$ follows at once from the definition of $\chi$.
For the uniqueness we observe that if two such clusters, say
${\mathcal{C}}'$ and ${\mathcal{C}}''$, would exist, then they would have to satisfy
$d({\mathcal{C}}', {\mathcal{C}}'') < L^{\chi(\xi)} = L^{\min\{m({\mathcal{C}}'),m({\mathcal{C}}'')\}}$, which
contradicts \eqref{2.five} by virtue of the assumption
${\mathcal{C}}', {\mathcal{C}}'' \in {\mathbf C}_\infty$.

We now construct a new environment $\widetilde\xi$ (depending on
$\xi$): when $\chi(\xi)=0$ we let $\widetilde \xi_i=\xi_i$ for all $i \ge 1$.
On the other hand, if $0<\chi(\xi)<\infty$ we set,
\begin{equation}
\label{2.39}
\widetilde \xi_i=\begin{cases}
 0 &\text{ if } i \le \omega(\mathcal{C}^*)\\ \xi_i &\text{ if } i >
\omega(\mathcal{C}^*).
\end{cases}
\end{equation}

We shall now show that
\begin{equation}
\chi(\widetilde\xi) = 0. \label{2.37}
\end{equation}

Of course we only have to check this in the case $0 <
\chi(\xi) < \infty$. We claim that $\mathbf{ C}_\infty(\widetilde \xi)$ is also
well defined an all clusters in
$\mathbf{ C}_\infty(\widetilde \xi)$
are also clusters in $\mathbf{C}_\infty(\xi)$ (which
are located in $[\omega(\mathcal{C}^*)+1, \infty)$) and the masses of such a
cluster in the two environments $\xi$ and $\widetilde \xi$ are the same.
To see this we simply run through the construction of the clusters in
$\cup_{\ell \ge 1}\mathbf{C}_\ell$ in the environment $\widetilde \xi$, until
there arises a difference between these this construction and the
construction in the environment $\xi$. More precisely, we apply
induction with respect to the level of the clusters. Clearly any
cluster of level 0 in $\widetilde \xi$ is simply a single point of $\Gamma(\xi)$
which lies in $[\omega(\mathcal{C}^*)+1, \infty)$, and has mass 1. This is also
a cluster of level 0 and mass 1
in $\xi$. Assume now that we already know that any cluster in $\widetilde
\xi$ of level at most $k$ is a cluster of $\xi$ of level $k$ and located
in $[\omega(\mathcal{C}^*)+1, \infty)$ and with the same mass in $\xi$ and $\widetilde \xi$.

Since $\widetilde\xi_i = 0$ for $i \le \omega(\mathcal{C}^*)$,
the span of any $(k+1)$-run in $\widetilde \xi$ has to be contained in
$[\omega(\mathcal{C}^*)+1,\infty)$. Therefore the span of any cluster of level $k+1$ in
environment $\widetilde \xi$ also has to be contained in $[\omega(\mathcal{C}^*)+1,\infty)$.
In addition, since the two environments $\xi$
and $\widetilde\xi$ agree in this interval, a difference in the
constructions or masses of some cluster of level $k+1$
can arise only because in $\xi$ there
is a $(k+1)$-run which contains clusters of level at most $k$ which lie in
$[\omega(\mathcal{C}^*)+1, \infty)$ as well as clusters which intersect
$[0, \omega(\mathcal{C}^*)]$.
But then these  clusters of level at most $k$ will be constituents of a single
$(k+1)$-cluster, $\mathcal{C} \in \bold C_\infty(\xi)$ say. Thus span($\mathcal{C}$)
has to contain points of both
$[0,\omega(\mathcal{C}^*)]$ and of $[\omega(\mathcal{C}^*)+1, \infty)$ in $\xi$.
Consequently, span($\mathcal{C})$ has to contain both points $\omega(\mathcal{C}^*)$
and $\omega(\mathcal{C}^*)+1$. Since $\omega(\mathcal{C}^*) \in \mathcal{C}^*$ we then have from
\eqref{2.y} that span($\mathcal{C}^*) \subset \text{span}(\mathcal{C})$ and $\mathcal{C}^* \ne
\mathcal{C}$ (because $\omega(\mathcal{C}^*)+1 \notin \text{span} (\mathcal{C}^*$)). But no such $\mathcal{C}$
can exist, because $\mathcal{C}^* \in \bold C_\infty$.

This establishes our last claim. Now, by definition of $\chi$,
\eqref{2.37} is equivalent to
\begin{equation}
d(\mathcal{C},0) \ge L^{m(\mathcal{C})} \label{2.38}
\end{equation}
for all clusters $\mathcal{C}$ in $\bold C_\infty(\widetilde\xi)$. In view of
our claim this will be implied by \eqref{2.38} for all
clusters $\mathcal{C}$ in $\mathbf{C}_\infty(\xi)$ located in $[\omega(\mathcal{C}^*),
\infty)$. Now, if $\mathcal{C}$ is such a cluster with $m(\mathcal{C}) \le m(\mathcal{C}^*)$,
then \eqref{2.38} holds, because, by virtue of \eqref{2.five},
\begin{equation*}
d(\mathcal{C},0)= \alpha(\mathcal{C}) \ge \alpha(\mathcal{C}) - \omega(\mathcal{C}^*) = d(\mathcal{C}, \mathcal{C}^*) \ge
L^{m(\mathcal{C})}.
\end{equation*}
On the other hand, if $m(\mathcal{C}) > m(\mathcal{C}^*) = \chi(\xi)$, then the
definition of $\chi$ shows that we have $\alpha(\mathcal{C}) \ge L^{m(\mathcal{C})}$.
This proves \eqref{2.38} in all cases, and therefore also proves
\eqref{2.37}.

We now have
$$
1 = \mathbb{P}_p(\chi(\xi) < \infty ) \le \mathbb{P}_p(\chi(\xi) = 0) + \sum_{n
=0}^\infty \mathbb{P}_p(\omega(\mathcal{C}^*) =n,\chi(\xi^{(n)})=0),
$$
where $\omega(\mathcal{C}^*)$ is as described above, and $\xi^{(n)}$  given by
\begin{equation}
\xi_i^{(n)} = \begin{cases} 0 &\text{ if } i\le n\\
\xi_i & \text{ if } i >n.
\end{cases}
\end{equation}

Thus, either $\mathbb{P}_p(\chi(\xi)=0 ) >0$ or there is some non-random $n
\in \mathbb{Z}_+$ for which $\mathbb{P}_p(\chi(\xi^{(n)})= 0) >0$. However,

\begin{eqnarray*}
 \mathbb{P}_p(\chi(\xi) = 0)  &\ge& P(\chi(\xi) = 0,\; \xi_i = 0 \text{
for }
0 \le i \le n)\\\nonumber
& =& \mathbb{P}_p(\chi(\xi^{(n)}) = 0,\; \xi_i = 0 \text{ for }
0 \le i \le n)\\\nonumber
& =&\mathbb{P}_p(\chi(\xi^{(n)}) = 0) \mathbb{P}_p(\xi_i = 0 \text{ for } 0 \le i \le n)
\end{eqnarray*}
(because $\xi^{(n)}$ is determined by the $\xi_i$ with $i > n$).
Thus $\mathbb{P}_p(\chi(\xi) = 0) > 0$ in all cases, concluding the proof.
\qed

\bigskip

\noindent {\bf Acknowledgements.}

The authors thank Yuval Peres and Leonardo T. Rolla for many useful discussions, and thank Lionel Levine for the computer simulations. B.N.B.L. thanks CBPF and IMPA, and V.S. and M.E.V. thank UFMG for hospitality and support. Most of this work was done while M.E.V. was senior researcher of CBPF.  B.N.B.L. is partially supported by CNPq grant 301844/2008-9. M.E.V. is partially supported by CNPq grant 302796/2002-9. V. S. is partially supported by CNPq grant 484801/2011-2.

V.S. and M.E.V. thank MSRI for hospitality and financial support during the final preparation of this paper.


\begin{thebibliography}{99}
\bibitem{ACHR} M.R. Allen, G.S.H. Cruttwell, K. E. Hare, J.-O. R\"oning (2007). {\em Dimensions of fractals in the large}.
Caos, Solitons \& Fractals {\bf 31}, 5--13.
\bibitem{BT} M. Barlow, S.J. Taylor (1992). {\em Defining Fractal subsets of $\ZZ^d$}.
Proc. London Math. Society {\bf 64}, 125--152.
\bibitem{BK} I. Benjamini, H. Kesten (1995). {\em Percolation of arbitrary words in $\{0,1\}^\NN$}.
Annals of Probability {\bf 23}, 1024--1060.
\bibitem{Gac} P. G\'acs (2004). {\em Compatible sequences and a slow Winkler percolation}. Combin. Probab. Comput. {\bf 6}, 815--856.
\bibitem{GLR} G.R. Grimmett, T.M. Liggett, T. Richthammer (2010). {\em  Percolation of arbitrary words in one dimension}. Random Struct
\& Algorithms {\bf 37}, 1, 85--99.
\bibitem{HPT} F. Harary, G. Prins, W.T. Tutte (1964). {\em The Number of Plane Trees.} Indag. Math. {\bf 26}, 319--327.
\bibitem{KSV} H. Kesten, V. Sidoravicius, M.E. Vares (2007). {\em Oriented percolation in dependent environment}.
Preprint.
\bibitem{KSZ} H. Kesten, V. Sidoravicius, Y. Zhang  (2001). {\em Percolation of arbitrary words on the closed
package graph of $\ZZ^2$}. Eletronic Journal of Probability {\bf 6}, 1--27.
\bibitem{Wi} P. Winkler  (2000). {\em Dependent percolation and colliding random walks}. Random Structures
and Algorithms. {\bf 16}, 1, 58--84.
\end{thebibliography}
\end{document}